\documentclass{amsart}

\usepackage{amsmath,amsthm,amssymb}

\usepackage{graphicx}

\usepackage{hyperref}

\title{Hyperbolic groupoids: metric and measure}
\author{Volodymyr Nekrashevych}

\thanks{The author was supported by 
  NSF grant DMS1006280}

\newcommand{\half}[1][\X]{\overline{\G_x^{#1}}}

\newcommand{\bdry}[1][\X]{\mathfrak{dG}_{#1}}

\newcommand{\coc}{\nu}
\newcommand{\arr}{\longrightarrow}

\newcommand{\G}{\mathfrak{G}}
\newcommand{\Gh}{\mathfrak{H}}
\newcommand{\pG}{\widetilde{\G}}
\newcommand{\be}{\mathsf{o}}
\newcommand{\en}{\mathsf{t}}
\newcommand{\R}{\mathbb{R}}
\newcommand{\Z}{\mathbb{Z}}
\newcommand{\X}{X}

\newcommand{\mS}{\mathcal{S}}
\newcommand{\wh}{\widehat}
\newcommand{\wt}{\widetilde}

\newcommand{\til}{\mathcal{T}}
\newcommand{\proj}{\mathsf{P}}
\newcommand{\M}{\mathcal{M}}
\newcommand{\C}{\mathbb{C}}

\newcommand{\hide}[1]{}

\newtheorem{theorem}{Theorem}[section]
\newtheorem{proposition}[theorem]{Proposition}
\newtheorem{corollary}[theorem]{Corollary}
\newtheorem{lemma}[theorem]{Lemma}

\theoremstyle{definition}
\newtheorem{defi}{Definition}[section]
\newtheorem{examp}{Example}[section]

\begin{document}

\maketitle

\begin{abstract}
We construct Patterson-Sullivan measure and a natural metric on the
unit space of a hyperbolic groupoid. In particular, this gives a new
approach to defining SRB measures on Smale spaces and Anosov flows using
Gromov hyperbolic graphs.
\end{abstract}

\tableofcontents

\section{Introduction}

In the previous paper~\cite{nek:hyperbolic} we gave
the basic definitions related to the notion of a hyperbolic groupoid,
and proved the duality theorem for them. Here we continue to study
hyperbolic groupoids, and define natural classes of metrics and
measures on their unit spaces.

Both constructions generalize different classical notions in
hyperbolic dynamics (for Gromov hyperbolic groups, Anosov flows, and
Smale spaces).
The metric is a generalization of what is sometimes called the \emph{visual metric} on
the boundary of the hyperbolic group
(see~\cite{gro:hyperb,ghys-h:gromov}). More on properties of this
metric see~\cite{mineyev:metric}. In the case of Anosov flows it is
known as a ``natural'' or ``dynamical'' metric (on the stable and
unstable leaves, or on the whole space). It was, for instance, defined
by D.~Fried for Smale spaces in~\cite{fried:naturalmetric}.

The measure constructed here is equivalent to the Hausdorff measure for the metric,
and is a direct generalization of the
Patterson-Sullivan measure on the boundary of a hyperbolic group
(see~\cite{patterson,sullivan:density,coornaert:psmeasure}). Applying
this generalization to hyperbolic groupoids
associated with Anosov flows and Smale spaces, we recover the
classical Bowen-Margulis, or Sinai-Ruelle-Bowen measures. We get in
this way a new approach to defining these measures: we represent the
stable and unstable leaves as boundaries of
Gromov hyperbolic graphs and then apply the Patterson-Sullivan
construction.

Both the metric and the measure depend on the choice of a concrete
\emph{Busemann (quasi-)cocycle} $\coc$ on the groupoid of germs. The Busemann cocycle will
play then the role of a logarithm of the derivative, in the sense that an
element $F$ of the pseudogroup multiplies the metric in a neighborhood
of a point $x$ roughly by
$e^{-\alpha\coc(F, x)}$ and the measure by $e^{-\beta\coc(F, x)}$ for some
positive numbers $\alpha$ and $\beta$. Here $(F, x)$ denotes the germ
of $F$ at $x$, and $\beta$ is the \emph{entropy} of the groupoid (with
respect to the cocycle $\coc$).

Different choices of the cocycle may be natural in different
situations. For example, for appropriate choices of the cocycle our
construction gives either the measure of maximal entropy
or the Hausdorff measure on the Julia set of a hyperbolic complex
rational function (see Example~\ref{ex:ratfunction}). We also prove
that the usual metric on $\C$ restricted to the Julia set is locally
bi-Lipschitz equivalent to
the visual metric for a Busemann cocycle on the corresponding groupoid.

The fact that Bowen-Margulis measure comes from the Hausdorff measures
associated with natural metrics on the stable and unstable leafs
(which follows now from Corollary~\ref{cor:Hausdorff}) was proved by
B.~Hasselblatt~\cite{hasselblatt:hausdorffmargulis}. In the case of a
geodesic flow on a negatively curved manifold this was proved by
U.~Hamenst\"adt~\cite{hamenstadt:BMm}.

The fact that the Bowen-Margulis measures come from Patterson-Sullivan
measures was known before only for geodesic flows,
see~\cite{sullivan:measure,kaimanovich:geodesicflow}. In some way, we generalize
this to arbitrary hyperbolic dynamical systems (quasi-flows).

Relation between Gromov hyperbolicity and expanding dynamical systems
(in particular, from the metric point of view) is a subject of~\cite{haisinskypilgrim}.

\subsection*{Overview of the paper}

In Section~\ref{s:hypgroupoids}, we give a short
overview of the notions developed in~\cite{nek:hyperbolic}. We remind
the basic notations of the theory of pseudogroups and groupoids of
germs, recall the notion of a log-scale, and review the main
definitions related to hyperbolic groupoids and Smale quasi-flows,
duality theory for hyperbolic groupoids, and properties of minimal
hyperbolic groupoids.

We define the metric on the space of units of the groupoid 
in Section~\ref{s:visual}. Every Busemann
cocycle $\coc:\G\arr\R$ determines a natural ``logarithmic scale'' on the boundary
of the Cayley graph equal to the associated Gromov product. Its value
$\ell(\xi_1, \xi_2)$ is equal to minimum of the value of $\coc$ along
a geodesic path connecting $\xi_1$ and $\xi_2$ in the Cayley graph of
$\G$. Using the Cayley graph of the dual groupoid $\G^\top$ instead, we get a
log-scale $\ell$ on the space of units of $\G$. For any sufficiently small
real number $\alpha>0$ we can find a metric $|x-y|$ on $\G^{(0)}$ such that
$c^{-1}e^{-\alpha\ell(x, y)}\le |x-y|\le ce^{-\alpha\ell(x, y)}$ for
some constant $c>1$. We call such a metric \emph{hyperbolic metric of
  exponent $\alpha$}.

We show (Proposition~\ref{pr:dilationvisual})
that for every germ $g\in\G$
there exists a neighborhood $U\in\pG$ of $g$ and a constant $c>1$ such
that
\[c^{-1}e^{-\alpha\coc(g)}\le\frac{|U(x)-U(y)|}{|x-y|}\le
ce^{-\alpha\coc(g)}.\]
Moreover, we show that in some sense this property completely characterizes the
hyperbolic metrics, see Theorem~\ref{th:visualmcriterion}. In particular, we show
that if $\pG$ acts by conformal maps on a compact subset of $\C$, then
the usual metric on $\C$ is a hyperbolic metric of exponent $1$ for
$\G$ with respect to the cocycle $\coc(F, x)=-\ln |F'(x)|$, see
Proposition~\ref{pr:natural}.

In ``Growth and Entropy'' we show that growth of cones (graded
by the cocycle $\coc$) in a
Cayley graph of a hyperbolic groupoid is exponential, and give lower
and upper estimates of the form $c e^{\beta n}$ on the growth, where
$\beta$ depends only on the pair $(\G, \coc)$, and is called the \emph{entropy}
of the graded groupoid $(\G, \coc)$. In fact, we prove more general
estimates, which can be used to construct Gibbs measures for
hyperbolic groupoids.

Patterson-Sullivan measures for hyperbolic groupoids are constructed
in Section~7. It is a generalization of
the classical construction, were we use the Cayley graphs of the dual groupoid
$\G^\top$ to construct the measure for the groupoid $\G$. The measure $\mu$
is characterized by the property that the Radon-Nicodim derivative
$\frac{dF_*\mu}{\mu}(x)$ is estimated from below and from above by
functions of the form $c e^{-\beta\coc(F, x)}$, where $\beta$ is
the entropy of $(\G, \coc)$.
We also prove that the Patterson-Sullivan measure is equivalent to the Hausdorff
measure of dimension $\beta/\alpha$ for the hyperbolic metric on
$\G^{(0)}$ of exponent $\alpha$ (see Corollary~\ref{cor:Hausdorff}).

Note that in all these results the map $\coc:\G\arr\R$ is a
\emph{quasi-cocycle}: the equality $\coc(g_1g_2)=\coc(g_1)+\coc(g_2)$
holds only up to an additive constant. In particular, we get only
upper and lower estimates on the Radon-Nicodim derivative of the
Patterson-Sullivan measure.

On the other hand, if $\coc:\G\arr\R$ is \emph{H\"older continuous}
with respect to the hyperbolic metric on $\G$ (a condition depending
only on $\G$ and $\coc$), then our results can be made sharper. This
case is analyzed in Section~8. We develop a duality theory for
H\"older continuous cocycles, and show that in this case there exists
a unique (up to a multiplicative constant) measure $\mu$ satisfying
\[\frac{dF_*\mu}{\mu}(x)=e^{-\beta\coc(F, x)}.\]
Moreover, in this case it is easy to construct an invariant measure
on the geodesic flow of $\G$. In all classical examples coming from
Anosov flows, Smale spaces, and hyperbolic rational functions the
natural cocycles are H\"older continuous, and our constructions
produce the classical Sinai-Ruelle-Bowen and Bowen-Margulis measures.

\section{Hyperbolic groupoids}
\label{s:hypgroupoids}
\subsection{Groupoids and Pseudogroups}

Here we give a short review of notions related to pseudogroups and
groupoids of germs. For more, see~\cite{nek:hyperbolic}.

A \emph{pseudogroup} $\pG$ acting on a space $\X$ is a collection of
homeomorphisms between open subsets of $\X$ which is closed under:
\begin{itemize}
\item compositions;
\item taking inverses;
\item restricting onto open subsets;
\item taking unions (if for a homeomorphism $F:U\arr V$ there exists a
  covering $\{U_i\}$ of $U$ by open sets such that $F|_{U_i}\in\pG$,
  then $F\in\pG$).
\end{itemize}
We also assume that the identical homeomorphism $\X\arr\X$ belongs to
$\pG$.

A \emph{germ} of an element of $\pG$ is the equivalence class of a pair $(F, x)$, where
$F\in\pG$, and $x$ belongs to the domain of $F$. Here two pairs $(F_1,
x)$ and $(F_2, x)$ are equivalent if there exists a neighborhood $U$
of $x$ such that $F_1|_U=F_2|_U$. The set of germs of $\pG$ is a
\emph{groupoid}, i.e., it is a small category of isomorphisms with
respect to the usual composition and taking inverses. We will denote
the groupoid of germs of $\pG$ by $\G$. For a germ
$g=(F, x)\in\G$ we denote
\[\be(g)=x,\quad\en(g)=F(x),\]
and call them \emph{origin} and \emph{target} of $g$.
Similarly, we will denote for $F\in\pG$ by $\be(F)$ and $\en(F)$ the
domain and range of $F$, respectively.
Germs of the identity homeomorphism $\X\arr\X$ are called \emph{units}
of the groupoid and are identified with the corresponding points of
$\X$. We will also denote the set of units of a groupoid $\G$ by
$\G^{(0)}$.

We use the following notation:
\[\G_A=\{g\in\G\;:\;\be(g)\in A\},\quad\G^B=\{g\in\G\;:\;\en(g)\in
B\},\]
and
\[\G_A^B=\G_A\cap\G^B,\quad\G|_A=\G_A^A.\]
We also denote by $\G^{(2)}$ the set $\{(g_1,
g_2)\;:\;\en(g_2)=\be(g_1)\}$ of \emph{composable pairs} of the
groupoid $\G$.

The groupoid of germs $\G$ of a pseudogroup $\pG$ has a natural
topology. Namely, a basis of topology is given by the collection of
open sets of the form
\[\{(F, x)\;:\;x\in\be(F)\}.\]
The groupoid $\G$ is \emph{topological} with respect to this topology,
i.e., multiplication $\G^{(2)}\arr\G$ and inversion $\G\arr\G$ are
continuous.

On the other hand, the pseudogroup $\pG$ is uniquely determined by the
topological groupoid $\G$. We say that $F\subset\G$ is a
\emph{bisection} if $\be:F\arr\be(F)$ and $\en:F\arr\en(F)$ are
homeomorphisms. Then every bisection $F$ determines a homeomorphism
from $\be(F)$ to $\en(F)$ by the rule
$\be(g)\mapsto\en(g)$ for $g\in F$. It is easy to see that such
homeomorphisms are elements of $\pG$ and that every element $F\in\pG$
defines a bisection $\{(F, x)\;:\;x\in\be(F)\}$. Hence, $\pG$ is the
pseudogroup of bisections of the groupoid of germs $\G$. We will
identify $F\in\pG$ with the corresponding bisection (which is a subset
of $\G$). We will use therefore terminology of pseudogroups and groupoids as
equivalent languages describing the same object.

Two units $x, y\in\G^{(0)}$ belong to one \emph{$\G$-orbit} if there
exists $g\in\G$ such that $x=\be(g)$ and $y=\en(g)$.
A subset $Y\subset\G^{(0)}$ is said to be a \emph{$\G$-transversal} if it
intersects every $\G$-orbit.

Suppose that $f:Y\arr\G^{(0)}$ is a local homeomorphism such that
$f(Y)$ is a $\G$-transversal. Then \emph{localization} $\pG|_f$ of $\pG$
is the pseudogroup generated by all lifts by $f$ of elements of $\pG$,
i.e., by homeomorphisms $F:U\arr V$ between open subsets of $Y$ such
that $f|_U$ and $f|_V$ are homeomorphisms, and $f|_V\circ F\circ
f|_U^{-1}\in\pG$. We write the germs of the localization $\pG|_f$
as triples $(x, g, y)$, where $g\in\G$, and $x, y\in Y$ are such that
$f(x)=\be(g)$ and $f(y)=\en(g)$.

In particular, if $\mathcal{U}=\{U_i\}_{i\in\mathcal{I}}$ is an open
covering of a $\G$-transversal, then the corresponding localization
$\G|_{\mathcal{U}}$ consists triples $(i, g,
j)\in\mathcal{I}\times\G\times\mathcal{J}$ such that $\be(g)\in U_i$
and $\en(g)\in U_j$, which are multiplied by the rule
\[(i_1, g_1, j_1)\cdot (i_2, g_2, j_2)=(i_2, g_1g_2, j_1),\]
where the product is defined if and only if $j_2=i_1$ and
$\en(g_2)=\be(g_1)$.

\begin{defi}
\label{def:equivalent}
Two groupoids of germs $\G_1, \G_2$ are said to be \emph{equivalent}
if there exists a pseudogroup $\pG$ acting on
$\G_1^{(0)}\sqcup\G_2^{(0)}$ such that $\G|_{\G_1^{(0)}}=\G_1$,
$\G|_{\G_2^{(0)}}=\G_2$, and every $\G$-orbit is a union of one
$\G_1$-orbit and one $\G_2$-orbit.
\end{defi}

Two groupoids of germs are equivalent if and only if they have
isomorphic localizations. 

A general definition of equivalence of topological groupoids (not only
groupoids of germs) is a bit more complicated,
see~\cite[2.2.2]{nek:hyperbolic} and references therein.

We will often deal with covers of compact subsets of $\G$ by elements
of $\pG$. The following statement is proved
in~\cite[Lemma~2.1.1]{nek:hyperbolic}.

\begin{lemma}
Let $\pG$ be a pseudogroup acting on a metric space. Let $C\subset\G$
be a compact set, and let $\mathcal{U}\subset\pG$ be a covering of
$C$. Then there exists $\epsilon>0$ such that for every $g\in C$ there
exists $U\in\mathcal{U}$ such that $g\in U$ and the
$\epsilon$-neighborhood of $\be(g)$ is contained in $\be(U)$.
\end{lemma}

If $\epsilon$ satisfies the conditions of the lemma for a covering
$\mathcal{U}$, then we say that $\epsilon$ is a \emph{Lebesgue's number}
of the covering. If $g\in U\in\pG$ and the $\epsilon$-neighborhood of
$\be(g)$ is contained in $\be(U)$, then we say that $g$ is
\emph{$\epsilon$-contained} in $U$.

\subsection{Logarithmic scales}

It will be convenient sometimes to work with the following
version of the notion of distance. For more details,
see~\cite[Section~1.1]{nek:hyperbolic}.

\begin{defi}
A \emph{log-scale} on a set $\X$ is a function
$\ell:\X\arr\X\arr\R\cup\{+\infty\}$ such that
\begin{enumerate}
\item $\ell(x, y)=\ell(y, x)$ for all $x, y\in\X$;
\item $\ell(x, y)=+\infty$ if and only if $x=y$;
\item there exists $\delta>0$ such that
for any $x, y, z\in\X$ we have
\[\ell(x, z)\ge\min(\ell(x, y), \ell(y, z))-\delta.\]
\end{enumerate}
\end{defi}

It is proved in~\cite[Proposition~1.1.1]{nek:hyperbolic} that for every log-scale
and for all sufficiently small numbers $\alpha>0$ there exists a
metric $|x-y|$ on $\X$ and a number $C>1$ such that
\[c^{-1}e^{-\alpha\ell(x, y)}\le |x-y|\le ce^{-\alpha\ell(x, y)}\]
for all $x, y\in\X$. We say in this case that $|x-y|$ is an
\emph{associated metric of exponent $\alpha$} for the log-scale.
Note that any two metrics associated with $\ell$ are H\"older
equivalent to each other. In particular, they define the same topology.

Accordingly to the definition of an associated metric, we say that a
map $f$ between two sets with log-scales is \emph{Lipschitz}
if there exists $c>0$ such that $\ell(f(x), f(y))\ge\ell(x, y)-c$ for
all $x, y$. A map $f$ is \emph{bi-Lipschitz} if it is invertible, and
the maps $f$ and $f^{-1}$ are Lipschitz. A map $f$ is \emph{H\"older}
if there exist constants $c>1$ and $\eta>0$ such that
$\ell(f(x), f(y))\ge c\ell(x, y)+\eta$ for all $x, y$.

Two log-scales $\ell_1, \ell_2$ on a set $\X$ are said to be
\emph{Lipschitz equivalent} if $|\ell_1(x, y)-\ell_2(x, y)|$ is
uniformly bounded for all $x, y$ such that $x\ne y$. They are
\emph{H\"older equivalent} if there exist constants $c>1$ and $\eta>0$
such that
\[c^{-1}\ell_1(x, y)-\eta\le\ell_2(x, y)\le c\ell_1(x, y)+\eta\]
for all $x, y$.

Let $\pG$ be a pseudogroup acting on a space $\G^{(0)}$. A
\emph{Lipschitz structure} on $\pG$ (or on the corresponding groupoid
of germs $\G$) is a log-scale on $\G^{(0)}$ such that every element of
$\pG$ is locally bi-Lipschitz (i.e., if every germ $g\in\G$ has a
bi-Lipschitz neighborhood $U\in\pG$) with respect to the
log-scale. More on Lipschitz structures on pseudogroups,
see~\cite[Section~2.5.1]{nek:hyperbolic}.

We will use the following notations. Let $F$ and $G$ be two
real-valued functions. We write $F\doteq G$ if the difference $|F-G|$
is uniformly bounded for all values of the variables. We will write
$F\asymp G$ if there exists a constant $c>1$ such that $c^{-1}F\le
G\le c F$ for all values of the variables.

\subsection{Hyperbolic groupoids}

We will present here a review of the notions related to hyperbolic groupoids. For more
details see~\cite{nek:hyperbolic}. We assume that the reader is familiar with the basic theory of
Gromov-hyperbolic graphs (otherwise,
see~\cite[Section~1.2]{nek:hyperbolic} and the references therein).

We say that a subset $\X$ of the set of units of a groupoid $\G$ is a
\emph{topological transversal} if $\X$ contains an open transversal
of the groupoid.

\begin{defi} A \emph{generating pair} $(S, \X)$ of a groupoid $\G$
is a compact subset $S\subset\G$ and a compact topological
transversal $\X$ such that
for every $g\in\G|_{\X}$ there exists $n$ such that
the set $\bigcup_{0\le k\le n}(S\cup S^{-1})^k$ is a neighborhood of $g$ in
$\G|_{\X}$.
\end{defi}

If $(S, \X)$ is a generating pair of $\G$ and $x\in\X$, then the
\emph{Cayley graph} $\G(x, S)$ is the directed graph with the
set of vertices $\G_x^{\X}$
in which there is an arrow from $g$ to $h$ whenever
there exists $s\in S$ such that $h=sg$.

\begin{defi}
\label{def:hyperbolic}
A groupoid of germs $\G$ is \emph{hyperbolic} if there exists a
compact generating pair $(S, \X)$, a metric $|\cdot|$ defined on a
neighborhood of $\X$, and numbers $\lambda\in (0, 1), \delta,
\Lambda, \Delta>0$ such that
\begin{enumerate}
\item each element of $\pG$ is locally Lipschitz;
\item each element $g\in S$ is a germ of a $\lambda$-contraction $F\in\pG$;
\item $\be(S)=\en(S)=\X$;
\item for every $x\in\X$ the Cayley graph $\G(x, S)$ is
  $\delta$-hyperbolic;
\item for every $x\in\X$ there exists $\omega_x\in\partial\G(x, S)$
  such that every directed path in the Cayley graph $\G(x, S^{-1})$ is
  a $(\Delta, \Lambda)$-quasigeodesic converging to $\omega_x$.
\end{enumerate}
\end{defi}

Here a \emph{$(\Delta, \Lambda)$-quasi-geodesic} is a (finite or
infinite)  sequence $v_0,
v_1, \ldots, $ of vertices such that
$|v_i-v_{i+1}|<\Delta$ for all
$i$, (where $|v_i-v_{i+1}|$ is the combinatorial distance in the
graph) and $|i-j|\le\Lambda |v_i-v_j|$ for
every pair of indices $i, j$.

The \emph{Busemann cocycle} $\beta_\omega$ on a Gromov-hyperbolic
graph $\Gamma$, where $\omega\in\partial\Gamma$, is given by
\[\beta_\omega(v_1, v_2)=\lim_{v\to\omega}(|v_1-v|-|v_2-v|),\]
where $v_1, v_2$ are vertices of the graph, $|\cdot|$ is the
combinatorial metric on the graph, and we choose any one of the partial limits
on the right-hand side. The number $\beta_\omega$ is uniquely
defined, up to an additive constant (which depends only on $\delta$).

\begin{defi}
\label{def:etaqcocycle}
An \emph{$\eta$-quasi-cocycle} on a groupoid $\G$ is a map
$\coc:\G|_{\X}\arr\R$, where $\X$ is a topological
transversal, such that
\begin{enumerate}
\item
for every $g\in\G|_{\X}$ there exists a
neighborhood $U$ of $g$ such that $|\coc(g)-\coc(h)|<\eta$ for all
$h\in U\cap\G_{\X}$;
\item $|\coc(g_1g_2)-\coc(g_1)-\coc(g_2)|<\eta$ for any composable
  pair $g_1, g_2$ of elements of $\G|_{\X}$.
\end{enumerate}
\end{defi}

A \emph{graded} groupoid is a groupoid $\G$ together with a
quasi-cocycle. Two quasi-cocycles
$\coc_1:\G|_{\X_1}\arr\R$ and $\coc_2:\G|_{\X_2}\arr\R$ define the
same grading (are \emph{strongly equivalent}) if there exists a
quasi-cocycle $\coc:\G|_{\X_1\sqcup\X_2}\arr\R$ such that
$\left|\coc(g)-\coc_1(g)\right|$ and
$\left|\coc(h)-\coc_2(h)\right|$ are bounded (where $g\in\G|_{\X_1}$
and $h\in\G|_{\X_2}$).

Two quasi-cocycles $\coc_1, \coc_2:\G|_{\X}\arr\R$ are \emph{coarsely
  equivalent} if there exist constants $\Lambda>1$ and $c>0$ such that
\[\Lambda^{-1}\coc_1(g)-c\le\coc_2(g)\le\Lambda\coc_1(g)+c\]
for all $g\in\G|_{\X}$.

One can check that the Busemann cocycle
$\beta_{\omega_x}(g_1, g_2)$ on the Cayley graph $\G(x, S)$
depends (up to an additive
constant) only on $g_1g_2^{-1}$. We get thus a quasi-cocycle
$\coc(g_1g_2^{-1})=\beta_{\omega_x}(g_1, g_2)$, which we will also
call a \emph{Busemann quasi-cocycle} on the groupoid $\G$. More
generally, any quasi-cocycle that is coarsely equivalent to $\coc$ will
be called a \emph{Busemann quasi-cocycle} of the groupoid $\G$. One can show
that any two Busemann quasi-cocycles (defined by different generating
pairs) are coarsely equivalent to each other (see~\cite[Proposition~3.5.1]{nek:hyperbolic}).

A \emph{graded hyperbolic groupoid} $(\G, \coc)$ is a hyperbolic
groupoid together with a strong equivalence class of a Busemann quasi-cocycle.

\begin{defi}
\label{def:positive}
We say that a subset $C\subset\G$ of a graded hyperbolic groupoid
$(\G, \coc)$ is \emph{positive} if for every $g\in C$ we have
$\coc(g)>2\eta$, where $\eta$ is as in
Definition~\ref{def:etaqcocycle}.

A subset $C\subset\G$ is \emph{contracting} if for every $g\in C$
there exists a contracting map
$U\in\pG$ such that $g\in U$.
\end{defi}

If $C$ is positive, then for every composable product $\ldots g_2g_1$
of elements of $C$ the path $g_1, g_2g_1, g_3g_2g_1, \ldots$ is a
quasi-geodesic path converging to a point of $\partial\G(x,
S)\setminus\{\omega_x\}$.

We denote
\[\partial\G_x=\partial\G(x, S)\setminus\{\omega_x\},\]
where $\partial\G(x, S)$ is the boundary of the hyperbolic graph
$\G(x, S)$.
One can show that $\partial\G(x, S)$, $\omega_x$, and $\partial\G_x$
do not depend on the choice of $(S, \X)$.
Denote also by $\half$ the space $\G_x^{\X}\cup\partial\G_x$, i.e., the
completion of the Cayley graph $\G(x, S)$ with the point
$\omega_x$ removed. Note that $\half$ does not depend on $S$ (but depends on
$\X$).

If $(S, \X)$ is a compact generating pair satisfying the
conditions of Definition~\ref{def:hyperbolic}, then we denote
\[T_x=\bigcup_{n\ge 0}S^n\cap\G_x,\]
i.e., $T_x$ is the set of elements of $\G_x$ representable as a
product of elements of $S$ (where inverses are not allowed).

Similarly, we denote $T_g=\bigcup_{n\ge 0}S^n\cdot g$
for $g\in\G$. We obviously have a bijection $x\mapsto x\cdot g$
between $T_{\en(g)}$ and $T_g$.

We denote by $\til_g$ the intersection of the closure of $T_g$
in $\half$ with $\partial\G_x$ for $x=\be(g)$. It is equal to the
set of points of $\partial\G_x$ that can be represented as infinite
products
\[\ldots g_2g_1g=\lim_{n\to\infty}g_n\cdots g_1\cdot g\]
for $g_i\in S$. We will denote
$\overline{T_g}=T_g\cup\til_g$.

The following proposition is proved in~\cite[Proposition~3.3.1]{nek:hyperbolic}.

\begin{proposition}
\label{pr:hyperbolicgenset} Let $(\G, \coc_0)$ be a graded
hyperbolic groupoid. Let
$\X$ be a compact topological $\G$-transversal.

Then there exist a compact generating set $S$ of $\G|_{\X}$, a
metric $|\cdot|$ on a neighborhood $\wh{\X}$ of $\X$, and an $\eta$-quasi-cocycle
$\coc:\G|_{\wh\X}\arr\Z$ strongly equivalent to $\coc_0$, such that
\begin{enumerate}
\item for every $g\in S$ we have $\coc(g)>3\eta$;
\item $\be(S)=\en(S)=\X$;
\item there exists $\lambda\in(0, 1)$ such that every
$g\in S$ has a $\lambda$-contracting neighborhood
$U\in\pG|_{\wh\X}$;
\item every element $g\in\G|_{\X}$ is equal to a
product of the form $g_n\cdots g_1\cdot (h_m\cdots h_1)^{-1}$ for
some $g_i, h_i\in S$.
\end{enumerate}
\end{proposition}

The following proposition is proved in the same way
as~\cite[Proposition~3.4.4]{nek:hyperbolic}.

\begin{proposition}
\label{pr:nbhdtil} Let $(S, \X)$ be a generating pair of $\G$
satisfying the conditions of Proposition~\ref{pr:hyperbolicgenset}.
Then there exists a compact set
$A\subset\G|_{\X}$ such that for every $h\in\G_{\X}$
there exists $a\in A$ such that $\overline{T_{ah}}$ is a
neighborhood of $\overline{T_h}$.
\end{proposition}

\subsection{Smale quasi-flows}

We present here definition of the notion of a \emph{Smale quasi-flow}
generalizing the classical notion of a \emph{Smale space}
(see~\cite{ruelle:therm,putnam}). More details can be found
in~\cite{nek:hyperbolic}.

Let $R$ be a topological space. A \emph{direct product structure} on
$R$ is given by a continuous map $[\cdot,
\cdot]:R\times R\arr R$ such that
\[[x, x]=x,\quad [[x, y], z]=[x, z],\quad [x,
[y, z]]=[x, z]\]
for all $x, y, z\in R$. One can show that if $[\cdot, \cdot]$ defines a
direct product structure, then we can find topological spaces $A, B$
and a homeomorphism $\pi:A\times B\arr R$ such that $[\pi(a_1, b_1),
\pi(a_2, b_2)]=\pi(a_1, b_2)$. Namely, we can take $A=\proj_1(R, x)$
and $B=\proj_2(R, x)$, where
\begin{equation}\label{eq:projRx}
\proj_1(R, x)=\{y\in R\;:\;[x, y]=x\},\qquad \proj_2(R, x)=\{y\in
R\;:\;[x, y]=y\},\end{equation}
and $\pi(y_1, y_2)=[y_1, y_2]$.

A \emph{local product structure} on a space $\X$ is given by a
covering (an \emph{atlas})
of $\X$ by open subsets $R_i$ (\emph{rectangles})
together with direct product
structures $[\cdot, \cdot]_{R_i}$ on them such that for any pair $R_i,
R_j$ of rectangles and for any $t\in R_i\cup R_j$ there exists a
neighborhood $U$ of $t$ and a direct product structure $[\cdot,
\cdot]_U$ on it such that $[x, y]_{R_i}=[x, y]_U$ and $[x,
y]_{R_j}=[x, y]_U$ whenever the corresponding expressions are defined.

If $\X$ is a space with a local product structure, then an open subset
$R\subset\X$ and a direct product structure $[\cdot, \cdot]_R$ on $R$
is a \emph{rectangle} of $\X$ if when we add it to an atlas of $\X$, we
get again an atlas of $\X$. In particular, we can define the maximal
atlas of a local product structure consisting of all rectangles of
$\X$.

We say that a metric $|x-y|$ \emph{agrees with a local product structure} on
$\X$ if for every rectangle $R=A\times B$ of $\X$ there exist metrics
$|\cdot|_A$ and $|\cdot|_B$ on $A$ and $B$ such that restriction of
$|x-y|$ onto $R$ is locally bi-Lipschitz equivalent to the metric
\[|(a_1, b_1)-(a_2, b_2)|_R=\max\{|a_1-a_2|_A, |b_1-b_2|_B\}.\]

A pseudogroup $\pG$ acting on a space $\X$ with a local product
structure \emph{preserves the local product structure} if for any germ
$(F, z)$ of $\pG$ there exist rectangles $R_i$ and $R_j$ such that
$z\in R_i$, $F(z)\in R_j$ and $F([x, y]_{R_i})=[F(x), F(y)]_{R_j}$ for
all $x, y$ belonging to a neighborhood of $z$. Note that if $\pG$
preserves a local product structure then for any germ $g\in\G$ of
$\pG$ there exist rectangles $R_{\be(g)}=A_{\be(g)}\times B_{\be(g)}$
and $R_{\en(g)}=A_{\en(g)}\times B_{\en(g)}$ and a
neighborhood $F\in\pG$ of $g$ such that $\be(F)=R_{\be(g)}$,
$\en(F)=R_{\en(g)}$, and the map $F:R_{\be(g)}\arr R_{\en(g)}$ can be
decomposed into a direct product of maps $A_F:A_{\be(g)}\arr
A_{\en(g)}$ and $B_F:B_{\be(g)}\arr B_{\en(g)}$. In particular, the
groupoid of germs $\G$ of $\pG$ has a local product structure in a
natural way. \emph{Projections} of the germ $g$ are the germs
$\proj_1(g)$ and $\proj_2(g)$ of $A_F$ and $B_F$, respectively, at the
points $a\in A$ and $b\in B$ such that $(a, b)=\be(g)$. We also denote
$A_F=\proj_1(F)$ and $B_F=\proj_2(F)$.

Let $\G$ be a groupoid of germs preserving a local product structure on
$\G^{(0)}$, and let $\coc$ be a quasi-cocycle defined on a restriction of
$\G$ onto a compact topological transversal $\X$. We say that $\coc$ agrees
with the local product structure if there exists an open covering
$\mathcal{R}$ of $\X$ by rectangles and a constant $c>0$ such that if
$\proj_i(g_1)=\proj_i(g_2)$ for some $g_1, g_2\in\G$ and $i\in\{1, 2\}$, then
$|\coc(g_1)-\coc(g_2)|<c$.

A groupoid $\G$ preserving a local product structure on $\G^{(0)}$ is
\emph{locally diagonal} if there exits a covering $\mathcal{R}$ of a
topological $\G$-transversal by open rectangles such that if for
$g\in\G$ either $\proj_1(g)$ or $\proj_2(g)$
is a unit, then $g$ is a unit.

\begin{defi}
A \emph{Smale quasi-flow} is a groupoid $\G$ together with an
$\eta$-quasi-cocycle $\coc:\G|_{\X}\arr\R$ and
a local product structure on $\G^{(0)}$ such that there exists a
compact generating pair $(S, \X)$, a metric $|\cdot|$ defined on a
neighborhood of $\X$, and a number $\lambda\in (0, 1)$ such that
\begin{enumerate}
\item the metric $|\cdot|$ and the quasi-cocycle $\coc$ agree with the
  local product structure;
\item $\pG$ acts by locally Lipschitz transformations with respect to $|\cdot|$;
\item $\be(S)=\en(S)=\X$, and $\coc(g)>3\eta$ for all $g\in S$;
\item for every $g\in S$ there exists a rectangular neighborhood
  $F\in\pG$ of $g$ such that restrictions of $F$ and $F^{-1}$ onto
 $\proj_1(\be(F), x)$ and $\proj_2(\be(F), x)$, respectively, are
  $\lambda$-contractions for all $x\in\be(F)$;
\item for every compact subset $C\subset\X$ and for every
  real number $k>0$ the closure of the set
  $\{g\in\G|_\X\;:\;|\coc(g)|\le k\}$ is compact;
\item the groupoid $\G$ is locally diagonal.
\end{enumerate}
\end{defi}

For definition of sets $\proj_i(R, x)$, see~\eqref{eq:projRx} on page~\pageref{eq:projRx}.
We denote $\proj_1$ and $\proj_2$ in the case of a Smale quasi-flow by
$\proj_+$ and $\proj_-$, respectively.

Let $(\G, \coc)$ be a Smale quasi-flow. Let $\mathcal{R}$ be a covering
of a topological $\G$-transversal by sufficiently small
rectangles. Consider localization $\G|_{\mathcal{R}}$
of $\G$ onto $\mathcal{R}$, and
denote by $\proj_+(\G)$ and $\proj_-(\G)$ the groupoids of germs of
pseudogroups generated by projections $\proj_+(F)$ and $\proj_-(F)$ of
rectangular elements of the pseudogroup $\pG|_{\mathcal{R}}$. It is
proved in~\cite{nek:hyperbolic} that groupoids $\proj_+(\G)$ and
$\proj_-(\G)$ are well defined up to equivalence of groupoids (i.e.,
their equivalence class does not depend on the choice of $\mathcal{R}$). They
are called \emph{Ruelle groupoids} of the quasi-flow $\G$.

We will need the following technical result of~\cite[Proposition~4.2.1]{nek:hyperbolic}.

\begin{proposition}
\label{prop:generatorsflow}
Every Smale quasi-flow is equivalent to a groupoid $\Gh$
satisfying the following properties.

The space of units $\Gh^{(0)}$ is a disjoint union of a finite
number of rectangles $W_1=A_1\times B_1, \ldots, W_n=A_m\times
B_m$.

There exists an open transversal $\X_0$ equal to the union of open
sub-rectangles $W_i^\circ=A_i^\circ\times B_i^\circ\subset R_i$
such that the closure of $R_i^\circ$ is compact. Denote by $\X$ the
union of closures of the rectangles $W_i^\circ$.

There exists a finite set $\mS$ of elements of the pseudogroup
$\wt\Gh$ such that
\begin{enumerate}
\item every $F\in\mS$ is a rectangle $A_F\times B_F=\proj_+(F)\times\proj_-(F)$;
\item for every $F\in\mS$ there exist $i, j\in 1, \ldots, n$ such that
$\be(F)\subset W_i$, $\en(F)\subset W_j$, $\be(A_F)=A_i$, $\en(B_F)=B_j$;
\item intersections of $\be(F)$ and $\en(F)$ with $\X$ are
non-empty;
\item $A_F$ and $B_F^{-1}$ are $\lambda$-contracting for some
$\lambda\in (0, 1)$;
\item $S=\{(F, x)\;:\;x, F(x)\in\X\}$ is a
generating set of $\Gh|_{\X}$ (i.e., $(S, \X)$ is a generating pair);
\item $\be(S)=\en(S)=\X$;
\item $\coc(g)>2\eta$ for all germs of elements of $\mS$;
\end{enumerate}
\end{proposition}

\subsection{Dual groupoid}

Let $(\G, \coc)$ be a graded hyperbolic groupoid.
Let $(S, \X)$ be a generating pair of $\G$ satisfying the conditions of
Proposition~\ref{pr:hyperbolicgenset} for the quasi-cocycle $\coc$.
Let $\mS$ be a finite covering of $S$ by contracting positive elements of $\pG$.

Let $A\subset\G$ be a compact set satisfying the conditions of
Proposition~\ref{pr:nbhdtil}.
Suppose also that for any two sequences $g_i, h_i$ of
germs of elements of $\mS$ an equality $\ldots g_2g_1\cdot g=\ldots
h_2h_1\cdot h$ for some $g, h$, $\be(g)=\be(h)\in\X$ implies
that for all sufficiently big $n$ there exists $m$ and $a\in A$ such that
$ag_n\cdots g_1g=h_m\cdots h_1h$. Existence of such a set $A$ follows
from hyperbolicity of the Cayley graphs of $\G$ and the fact that all
directed paths in $\G(x, S)$ are quasi-geodesics.

Find then a finite covering $\mathcal{A}=\{U\}$ of $A$ by bi-Lipschitz
elements of $\pG$. Let $\wh A$ be the set of germs of the
elements of $\mathcal{A}$.

The following lemma is proved in~\cite[Lemmas~3.6.3, and~4.6.1]{nek:hyperbolic}.

\begin{lemma}
\label{lem:deltabijection}
Let $\epsilon$ be a common Lebesgue's number of the coverings $\mS$, $\mathcal{A}$, and $\mathcal{A}^{-1}$ of $S$, $A$, and $A^{-1}$, respectively.
There exists $0<\delta_0<\epsilon$ such that the following condition is satisfied.

Let $U_i, V_i$, $i=1, 2, \ldots$ be finite or infinite sequences of
elements of the set $\mS\cup\mathcal{A}$ in which at most one
element belongs to $\mathcal{A}$. Let $|x-y|<\delta_0$ for $x,
y\in\X$, the $\epsilon$-neighborhoods of $U_i\cdots U_1(x)$ and
$V_i\cdots V_1(x)$ are contained in $\be(U_{i+1})$ and $\be(V_{i+1})$, respectively.
Then an equality
\[(\ldots U_2U_1, x)=(\ldots V_2V_1, x)\]
of finite or infinite products of germs implies
\[(\ldots U_2U_1, y)=(\ldots V_2V_1, y).\]
\end{lemma}

Fix $\delta_0$ satisfying the conditions of
Lemma~\ref{lem:deltabijection}. Suppose that $g\in\G|_{\X}$ and
$h\in\G$ are such that $|\en(g)-\en(h)|<\delta_0$. For a finite or
infinite product
$\xi=\ldots g_2g_1g\in\overline{T_g}$, where $g_i\in S$,
find elements $U_i\in\mS$ such that $g_i$ is $\epsilon$-contained in $U_i$.
Define then
\begin{equation}\label{eq:rgh}R_g^h(\xi)=\ldots U_2U_1\cdot h.\end{equation}
By Lemma~\ref{lem:deltabijection}, $R_g^h(\xi)$ depends only on
$g$, $h$, and $\xi$ (and does not depend on the choice of the generators
$g_i$ or the choice of the elements $U_i$). Note that $R_g^h(\xi)
\notin\G|_{\X}$ in general (even for $\xi\in\G_x^{\X}$).

For every $h\in\G$ we have a natural homeomorphism $\xi\mapsto\xi\cdot h$ from $\partial\G_{\en(h)}$ to $\partial\G_{\be(h)}$ defined by
\[\ldots g_2g_1\cdot g\mapsto \ldots g_2g_1\cdot gh.\]
Note that every germ of this homeomorphism is also a germ of
a transformation of the form $R_g^{gh}$ for some $g\in\G_{\en(h)}$.

The \emph{natural log-scale} (the \emph{Gromov product}) $\ell$ on $\half$ is defined by the
condition that $\ell(\xi_1, \xi_2)$ is equal to the minimal value of
$\coc$ on a geodesic path connecting $\xi_1$ to $\xi_2$. It is a
log-scale, which is well defined, up to bi-Lipschitz equivalence, by
the strong equivalence class of $\coc$.

The proof of the following proposition is straightforward.

\begin{proposition}
\label{pr:dilationestimate} Let $g\in\G_{\X}$ and $h\in\G$ are
such that $R_h^g$ is defined. Then the map $R_h^g$ is bi-Lipschitz with
respect to the natural log-scale. Moreover,
there exists a constant $c>0$ (not depending on $g$ and $h$) such that
\[\ell(\xi_1, \xi_2)+\coc(g)-\coc(h)-c\le
\ell(R_h^g(\xi_1), R_h^g(\xi_2))\le\ell(\xi_1, \xi_2))+\coc(g)-\coc(h)+c\]
for all $\xi_1, \xi_2\in\overline{T_h}$.
It particular, $R_h^g$ is a homeomorphism between $\til_h$ and $R_h^g(\til_h)$.
\end{proposition}

Here $\ell$ is defined using a generating set of $\G|_{\X_2}$ where
$\X_2$ contains $\be(\mS)\cup\en(\mS)$.

It is proved in~\cite[Theorem~4.3.1]{nek:hyperbolic} that the disjoint union
$\partial\G=\bigcup_{x\in\G^{(0)}}\partial\G_x$ has a natural topology
and a local products structure coming from the maps
$R_g^h$. Namely, if $U\in\pG$ is a sufficiently small neighborhood of
an element $g\in\G$, and $\xi$ is an interior point of $\til_g$, then
a neighborhood of $\xi$ in $\partial\G$ is the set
\[R_{g, U}=\bigcup_{h\in U}R_g^h(\til_g^\circ),\]
where $\til_g^\circ$ is the interior of $\til_g\subset\G_{\be(g)}$.
The set $R_{g, U}$ is naturally homeomorphic to
$\be(U)\times\til_g^\circ$, where the homeomorphism is given by the
map
\begin{equation}\label{eq:holonomy}
(x, \zeta)\mapsto [x, \zeta]_U:=R_g^{(U, x)}(\zeta).
\end{equation}
These direct product decompositions agree with each other, and we get
in this way a local product structure and topology on
$\partial\G$. For more details, see~\cite[Section~3.7]{nek:hyperbolic}.

Every $U\in\pG$ defines a local homeomorphism of $\partial\G$ with domain
$\partial\G_{\en(U)}$ and range $\partial\G_{\be(U)}$,
mapping $\xi\in\partial\G_{U(x)}$ for $x\in\en(U)$ to $\xi\cdot (U,
x)$. The groupoid of germs of the pseudogroup generated by such maps
is called the \emph{geodesic quasi-flow} of $\G$. Its elements can be
written as pairs $(\xi, g)$, where $\xi\in\partial\G_{\en(g)}$, $\be(\xi, g)=\xi$, and
$\en(\xi, g)=\xi\cdot g$. We denote the geodesic flow by
$\partial\G\rtimes\G$.

The following theorem is proved in~\cite[Theorem~4.6.2]{nek:hyperbolic}.

\begin{theorem}
\label{th:holonomies}
The space $\bdry$ of germs of restrictions of the maps $R_g^h$,
for $g\in\G|_{\X}, h\in\G$ onto open subsets of the disjoint union
$\bigsqcup_{x\in\X}\partial\G_x$ is a
groupoid (i.e., is closed under taking compositions and inverses), and
depends only on $\G$ and $\X$. 
\end{theorem}

The dual groupoid $\G^\top$ of a hyperbolic groupoid $\G$ is defined
in~\cite{nek:hyperbolic} as the projection
$\proj_-(\partial\G\rtimes\G)$.
It is also shown in~\cite[Section~4.6]{nek:hyperbolic} that this
definition is equivalent to the following.

\begin{defi}
\label{def:dual}
Let $\G$ be a hyperbolic groupoid. The \emph{dual groupoid} $\G^\top$
is any groupoid equivalent to $\bdry$.
\end{defi}

The groupoid $\bdry$ is not second countable, but it is equivalent to
a second countable groupoid.

\subsection{Minimal hyperbolic groupoids}

Let $\G$ be a hyperbolic groupoid and let $(S, \X)$ be its
generating pair. We
say that a Cayley graph $\G(x, S)$ is \emph{topologically mixing} if
for every point $\xi\in\partial\G_x$ and every neighborhood $U$ of
$\xi$ in $\half$ the set of accumulation points
of $\en(U\cap\G_x^{\X})$ contains the interior of $\X$.

The following description of topologically mixing hyperbolic groupoids
is given in~\cite[Proposition~4.7.1]{nek:hyperbolic}.

\begin{proposition}
\label{pr:minimal}
Let $\G$ be a hyperbolic groupoid. Then the following conditions are
equivalent.
\begin{enumerate}
\item Some Cayley graph of $\G$ is topologically mixing.
\item Every Cayley graph of $\G$ is topologically mixing.
\item Every $\G$-orbit is dense in $\G^{(0)}$.
\end{enumerate}
\end{proposition}

\begin{defi}
We say that a hyperbolic groupoid $\G$ is \emph{minimal}
if it satisfies the equivalent conditions of Proposition~\ref{pr:minimal}.
\end{defi}

For the proof of the next proposition and theorem, see~\cite[Section~4.7]{nek:hyperbolic}.

\begin{proposition}
\label{pr:dualgroupoid}
If $\G$ is minimal, then the groupoid $\bdry[x]$ equal to
restriction of $\bdry$ onto $\partial\G_x$ is equivalent to the
groupoid $\bdry$, and $\G^\top$ is also minimal.
\end{proposition}

The following theorem is one of central results
of~\cite{nek:hyperbolic}.

\begin{theorem}
\label{th:duality}
Let $\G$ be a minimal hyperbolic groupoid. Then
the groupoid $\G^\top$ is hyperbolic, and the groupoid $(\G^\top)^\top$
is equivalent to $\G$.
\end{theorem}

This duality theorem holds not only for minimal groupoids, but for all
hyperbolic groupoids whose geodesic flow is locally diagonal.

\section{Hyperbolic metric}
\label{s:visual}

It is proved in~\cite[Proposition~3.5.1]{nek:hyperbolic}
that the coarse equivalence class of a Busemann cocycle
on a hyperbolic groupoid $\G$ is uniquely
determined by the topological groupoid $\G$. A hyperbolic groupoids
together with a \emph{strong} equivalence class of the cocycle is
a \emph{graded hyperbolic groupoid}. In different situations
different gradings are natural.

\begin{examp}
\label{ex:hypratfunc}
Let $f$ be a hyperbolic complex rational function. Then $f$ is
expanding with respect to a Riemanian metric on a neighborhood of the
Julia set of $f$, hence the groupoid $\mathfrak{F}$ generated by the germs of the
action of $f$ on its Julia set is hyperbolic, where the grading
$\coc:\mathfrak{F}\arr\Z$ is given by the degree of germs. Namely,
every element of $\mathfrak{F}$ is a composition $g=(f^n, x)^{-1}\circ
(f^m, y)$ for some points $x, y$ of the Julia set. We define then
$\coc(g)=n-m$. The Cayley graphs of $\mathfrak{F}$ are regular trees
and $\coc$ is equal to the Busemann cocycle associated with the point
of its boundary given by the path $x, f(x), f^2(x), \ldots$. This
example is discussed in detail in~\cite{nek:hyperbolic}.

On the other hand, it is easy to see that the map
\begin{equation}\label{eq:derivative}\coc_1((F, x))=-\ln|F'(x)|,\end{equation}
for $F\in\wt{\mathfrak{F}}$, is a cocycle coarsely equivalent to
$\coc$. Consequently, $(\mathfrak{F}, \coc_1)$ is also a graded hyperbolic
groupoid.
\end{examp}

If $(\G, \coc)$ is a graded hyperbolic groupoid, then
$(\partial\G\rtimes\G, \wt\coc)$ is a graded groupoid where
$\wt\coc(g, \xi)=\coc(g)$ is the lift of the quasi-cocycle $\coc$ to
the geodesic quasi-flow. The graded groupoid  $(\partial\G\rtimes\G,
\wt\coc)$ is uniquely determined by the graded
groupoid $(\G, \coc)$.

On the other hand, by~\cite[Theorem~4.4.1]{nek:hyperbolic}, if
$(\mathfrak{H}, \coc)$ is a Smale quasi-flow, then its
projections $\proj_+(\mathfrak{H})$ and $\proj_-(\mathfrak{H})$
onto the stable and unstable directions are hyperbolic, and there exist
quasi-cocycles $\coc_+$ and $\coc_-$ on $\proj_+(\mathfrak{H})$ and
$\proj_-(\mathfrak{H})$ such that the functions
$|\coc_+(\proj_+(g))-\coc(g)|$ and $|\coc_-(\proj_-(g))+\coc(g)|$
are uniformly bounded.

We get hence the following summary of the above facts.

\begin{proposition}
\label{pr:wtcoccoctop}
Let $(\G, \coc)$ be a minimal graded hyperbolic groupoid. There
exist unique, up to strong equivalence, quasi-cocycles $\wt\coc$ and
$\coc^\top$ on $\partial\G\rtimes\G$ and $\G^\top$ such that
$|\coc(\proj_+(g))-\wt\coc(g)|$ and
$|\coc^\top(\proj_-(g))+\wt\coc(g)|$ are uniformly bounded.
\end{proposition}

It follows directly from the definitions that for every germ $(R_g^h,
\xi)$ we have
\begin{equation}\label{eq:coctopform}
\coc^\top(R_h^g, \xi)\doteq|\coc(g)-\coc(h)|.
\end{equation}

\begin{defi}
\label{def:wtcoccoctop}
Let $(\G, \coc)$ be a graded hyperbolic groupoid with locally diagonal
geodesic quasi-flow. Then the \emph{dual graded groupoid} $(\G, \coc)^\top$
is the groupoid $(\G^\top, \coc^\top)$ where $\G^\top$ is the
hyperbolic groupoid dual to $\G$, and the quasi-cocycle $\coc^\top:\G^\top\arr\R$ is equal
to the projection of the quasi-cocycle $-\wt\coc:\partial\G\rtimes\G\arr\R$,
where $\wt\coc$ is the lift of $\coc$.
\end{defi}

Let now $(\G, \coc)$ be a minimal graded hyperbolic groupoid.
The space of units of $\G^\top$ is locally homeomorphic to boundaries
of the Cayley graphs $\G(x, S)$ of $\G$. The quasi-cocycle $\coc$
defines a natural log-scale on $\partial\G_x$ by rule
that $\ell_\coc(\xi_1, \xi_2)$ is equal to the minimal value of $\coc$
along a geodesic path in $\G(x, S)$ connecting $\xi_1$ to $\xi_2$.

This log-scale satisfies an estimate (see Proposition~\ref{pr:dilationestimate})
\begin{equation}
\label{eq:dilationgf}
\ell_\coc(R_h^g(\xi_1), R_h^g(\xi_2))
\doteq\ell_\coc(\xi_1, \xi_2)+\coc(g)-\coc(h)
\end{equation}
for all $\xi_1, \xi_2\in\overline{T_h}$.

Consequently, we get a Lipschitz structure on $\G^\top$, which is
uniquely determined (up to
bi-Lipschitz equivalence) by the grading $\coc$. We will call it the
\emph{hyperbolic log-scale associated with $\coc$}. Since $\coc$ and
$\coc^\top$ uniquely determine each other, we will also say that the
defined log-scale is associated with $\coc^\top$ (if there is no
confusion on which of the groupoids $\G$ and $\G^\top$ the log-scale
is defined).

Estimates~\eqref{eq:dilationgf} and~\eqref{eq:coctopform} immediately
imply the following proposition.

\begin{proposition}
\label{pr:dilationdual}
Let $\ell$ be the hyperbolic log-scale on $\G$ associated with the grading
$\coc$ of the hyperbolic groupoid $(\G, \coc)$. Then there exists a
constant $c$ such that for every
$g\in\G$ there exists a neighborhood $U\in\pG$ such that for any two
points $x, y\in\be(U)$ we have
\[|\ell(U(x), U(y))-(\ell(x, y)+\coc(g))|<c.\]
\end{proposition}

Recall that we say that a metric $|\cdot|$ is \emph{associated} with a log-scale
$\ell$ if there exists a constant $\alpha>0$ such that
\[|x-y|\asymp e^{-\alpha\cdot\ell(x, y)}\]
for all pairs of points $x, y$. We call $\alpha$ the \emph{exponent}
of the associated metric. For every log-scale there exists a constant
$\alpha_0$ such that an associated metric exists for every positive exponent
$\alpha<\alpha_0$.

\begin{defi}
A \emph{(hyperbolic) metric of exponent $\alpha$} associated with the
quasi-cocycle $\coc$ is a metric of exponent $\alpha$
associated with the hyperbolic log-scale $\ell_\coc$.
\end{defi}

The hyperbolic metric is called sometimes \emph{visual}.
Note that, by definition, a metric locally bi-Lipschitz equivalent to a
hyperbolic metric of exponent $\alpha$ is also a hyperbolic metric of exponent
$\alpha$.

Thus, a grading $\coc$ of a hyperbolic groupoid determines for every
positive sufficiently close to zero number $\alpha$ a unique locally
bi-Lipschitz class of hyperbolic
metrics. Proposition~\ref{pr:dilationdual} is reformulated then as
follows.

\begin{proposition}
\label{pr:dilationvisual}
Let $|\cdot|$ be a hyperbolic metric of exponent $\alpha$ on a graded hyperbolic
groupoid $(\G, \coc)$. Then there exists a constant $c>1$ such that
for every $g\in\G$ there exists a neighborhood $U\in\pG$ such that for
every pair of different points $x, y\in\be(U)$ we have
\[c^{-1}e^{-\alpha\cdot\coc(g)}\le
\frac{|U(x)-U(y)|}{|x-y|}\le ce^{-\alpha\cdot\coc(g)}.\]
In other words, the quasi-cocycle $\coc(g)$ is proportional, up to an additive
constant, to the logarithm of the scaling factor of the germ $g$.
\end{proposition}

We will need a more precise version of the last proposition.

\begin{proposition}
\label{pr:positivedilation}
Let $|\cdot|$ be a hyperbolic metric of exponent $\alpha$. Let $S$ be a
compact positive subset of $\G$. Let $\mS$ be a finite
covering of $S$ by $\lambda$-contracting positive elements of $\pG$,
where $\lambda\in (0, 1)$ is a fixed constant.

Let $\epsilon>0$ be sufficiently small. 
Then there exists a constant $c>1$ such that
if $g_1g_2\cdots g_n$ is a product of elements of $S$, and $F_i\in\mS$
are such that $g_i$ is $\epsilon$-contained in $F_i$, then for any two points $x,
y\in\be(F_1\cdots F_n)$
on distance less than $\epsilon$ from $\be(g_n)$ we have
\begin{equation}\label{eq:dilationexact}
c^{-1}e^{-\alpha\coc(g_1\cdots g_n)}|x-y|\le |F_1\cdots
F_n(x)-F_1\cdots F_n(y)|\le
ce^{-\alpha\coc(g_1\cdots g_n)}|x-y|.
\end{equation}
\end{proposition}

\begin{proof}
Note that if $\mS'$ is a finite covering of $S$ subordinate to a
covering $\mS$, then the statement of the proposition holds for $\mS$
if and only if it holds for $\mS'$. Since any two open coverings of
$S$ have a common finite subordinate covering, if the statement is
true for some covering $\mS$, then it is true for all coverings.

Let $(S_1, \X_1)$ be a generating pair of $\G$ satisfying
Proposition~\ref{pr:hyperbolicgenset}
and let $S\subset\G$ be any positive contracting set. Assume that $S_1$
satisfies the conditions of the proposition. Let us show that it is
satisfied for $S$.

Let $(S_2, \X_2)$ be a generating pair of $\G$ such that
$S\subset\G|_{\X_2}$ and $\X_1\subset\X_2$. Then $\G_x^{\X_1}$ is a
net in $\G_x^{\X_2}$. Consequently, every element of $\G_{\X_2}$ can
be represented as a product $a_1ga_2$, where $g\in\G_{\X_1}$ and
$a_1$, $a_2$ belong to a fixed compact set $Q_1\subset\G$.

Every path corresponding to a finite or infinite
product $\cdots g_2g_1$ of elements
of $S$ is a quasi-geodesic (in a uniform way in the corresponding
Cayley graph $\G(\be(g_1), S_2)$). If the path is infinite, then it converges
to a point of $\partial\G_{\be(g_1)}$. Since $S_1$ satisfies the
conditions of Proposition~\ref{pr:hyperbolicgenset}, there
exists a compact set $Q\subset\G$ such that every product $g_n\cdots
g_1\in S^n$ can be represented in the form $a_1\cdot s_m\cdots s_1\cdot a_2$
for $a_1, a_2\in Q$ and $s_i\in S_1$.

Let $\mathcal{Q}$ be a finite covering of $Q$ by bi-Lipschitz elements of
$\pG$, and let $\mathcal{S}\subset\pG$ be a finite covering of $S\cup S_1$ by
contractions.
Then there exists $\epsilon>0$ such that
if we have $g_n\cdots g_1=h_0s_m\cdots s_1h_1$ for $g_i\in S$, $s_i\in
S_1$, $h_i\in Q$, and if $G_i\in\mS$, $H_i\in\mathcal{Q}$, and
$U_i\in\mS$ are such that $g_i$, $h_i$, and $s_i$ are
$\epsilon$-contained in $G_i$, $H_i$, and $U_i$, respectively, then
the compositions $G_n\cdots G_1$ and $H_0U_m\cdots U_1H_1$ coincide on the
$\epsilon$-neighborhood of $\be(g_1)=\be(h_1)$,
see~\cite[Corollary~2.4.2]{nek:hyperbolic}.
It follows that the proposition holds for $S$.

Consequently, it is enough to prove the proposition for any
groupoid equivalent to $\G$.

We can represent $\G$ as projection $\proj_+(\partial\G\rtimes\G)$
of the geodesic quasi-flow
and use a generating set equal to projection of the generating set
satisfying conditions of Proposition~\ref{prop:generatorsflow}. Then the
statement of the proposition will follow from~\eqref{eq:dilationgf} on page~\pageref{eq:dilationgf}
and the fact that the quasi-cocycle
$\wt\coc:\partial\G\rtimes\G\arr\R$ agrees with the local products structure.
\end{proof}

\begin{corollary}
\label{cor:cocdifc}
Let $S$ be a positive compact subset of a hyperbolic groupoid $(\G,
\coc)$, let $\mS$ be a finite covering of $S$ by
positive contracting elements of $\pG$.
Then there exist $c>0$ and $\epsilon>0$ such that for any product
$U=F_1\cdots F_n$ of elements of $\mS$ and for any two germs $g_1,
g_2$ of $U$ such that $|\be(g_1)-\be(g_2)|<\epsilon$ we have
\[|\coc(g_1)-\coc(g_2)|<c.\]
\end{corollary}

\begin{proof}
There exist $\epsilon>0$ and $c$ such that for every
product $U=F_1\cdots F_n$ of elements of $\mS$ and every pair $x_1,
x_2\in\be(U)$ for $F_i\in\mS$,
such that $|x_1-x_2|<\epsilon$ we
have $c^{-1}e^{-\alpha\coc(U, x_i)}|x_1-x_2|<|U(x_1)-U(x_2)|<
ce^{-\alpha\coc(U, x_i)}|x_1-x_2|$ for $i=1,2$. It implies
$|\coc(U, x_1)-\coc(U, x_2)|<2\alpha^{-1}\ln c$.
\end{proof}

\begin{theorem}
\label{th:visualmcriterion}
Let $(\G, \coc)$ be a graded hyperbolic groupoid, let
$(S, \X)$ be a generating pair satisfying the conditions of
Proposition~\ref{pr:hyperbolicgenset} (for an arbitrary metric $|\cdot|$).
A metric $|\cdot|_1$ defined on a neighborhood of $\X$ is a hyperbolic
metric of exponent $\alpha$ if and only if there exists a finite
covering $\mS$ of $S$ by positive
elements of $\pG$ such that for every product $F_1\cdots F_n$ of elements
of $\mS$ we have
\[|F_1\cdots F_n(x)-F_1\cdots F_n(y)|_1\asymp e^{-\alpha\coc(g)} |x-y|_1,\]
for all $x, y\in\be(F_1\cdots F_n)$,
where $g$ is any germ of $F_1\cdots F_n$ and the coefficients in the
estimate do not depend on $n$, $F_i$, $g$, $x$, and $y$.
\end{theorem}

\begin{proof}
Proposition~\ref{pr:dilationvisual} implies the `if' part of the theorem.
In order to prove the theorem in the other direction, is enough to
show that if $|\cdot|_1$ and $|\cdot|_2$ are metrics satisfying
the conditions of the theorem, then they are locally bi-Lipschitz
equivalent.

There exists a covering $\mS$ of $S$ satisfying the condition of the
theorem for both metric $|\cdot|_i$. Let $\epsilon$ be a common Lebesgue's
number of the covering $\mS$ for both metrics.

Let $x\in\X$ be an arbitrary point. Let $\Delta$ be an upper bound
on the value of the cocycle $\coc$ on elements of $S$ .
For every $n$ there exists a
sequence $g_1, \ldots, g_k$ of elements of $S$ such that
$\en(g_1\cdots g_k)=x$, and $n\le\coc(g_1\cdots
g_k)<n+\Delta+\eta$. Let $G_i\in\mS$ be such that $g_i$ is
$\epsilon$-contained in $G_i$. Then, for any $i=1,2$, if
\[|x-y|_i<c^{-1}e^{-\alpha(n+\Delta+\eta)}\epsilon=c^{-1}e^{-\alpha(\Delta+\eta)}\epsilon\cdot
e^{-\alpha n},\] then $(G_1\cdots G_k)^{-1}(y)$ is defined. Here $c>1$
is a sufficiently big constant.

On the other hand, if $( G_1\cdots G_k)^{-1}(y)$ is defined,
then $|x-y|_i<cD e^{-\alpha n}$, where $D$ is an upper bound on
the diameters of the sets $\be(G_i)$.

Since $|x-y|_1<c^{-1}e^{-\alpha(\Delta+\eta)}\epsilon\cdot
e^{-\alpha n}$ for
$n=\left\lfloor\frac{\ln\left(ce^{\alpha(\Delta+\eta)}\epsilon^{-1}\cdot
      |x-y|_1\right)}{-\alpha}\right\rfloor$,
we have for all $x, y$, such that $|x-y|$ is small enough
\begin{multline*}
|x-y|_2<cD\exp\left(-\alpha\left\lfloor\frac{\ln
\left(ce^{\alpha(\Delta+\eta)}\epsilon^{-1}\cdot |x-y|_1\right)}{-\alpha}\right\rfloor\right)<\\
cD\exp\left(\ln \left(ce^{\alpha(\Delta+\eta)}\epsilon^{-1}\cdot
|x-y|_1\right)+\alpha\right)=c^2D \epsilon^{-1}
e^{\alpha(\Delta+\eta+1)}\cdot |x-y|_1,
\end{multline*}
hence $|\cdot|_1$ and $|\cdot|_2$ are locally bi-Lipschitz equivalent.
\end{proof}

As an example of application of the previous theorem, consider the
following.

\begin{proposition}
\label{pr:natural}
Let $\pG$ a  pseudogroup acting on a subset $\X$ of $\C$ by
biholomorphic maps. Define $\coc(F, z)=-\ln|F'(z)|$ and suppose that
$(\G, \coc)$ is
hyperbolic. Then the usual metric $|z_1-z_2|$ on $\C$ is a
hyperbolic metric on $\X$ of exponent $1$.
\end{proposition}

In particular, if $\pG$ is the groupoid generated by the restriction
of a hyperbolic complex rational function onto its Julia set, then the
usual metric on $\C$ (restricted to the Julia set) is a hyperbolic
metric for the groupoid $(\G, \coc)$, where $\coc$ is as in
Proposition~\ref{pr:natural}.

\begin{proof}
Let $S$ be a compact subset of $\G$ such that $\coc(s)>0$ for all
$s\in S$. Then $\coc(s)$ for $s\in S$ is bounded from below by a
positive constant. Let $\mS$ be a finite open covering of
$S$ by relatively compact extendable $\lambda$-contracting elements of
$\pG$, where $0<\lambda<1$. Here an element $F\in\pG$ is
\emph{extendable} if there exists $F'\in\pG$ containing the closure of $F$.

Since the set $\mS$ is finite, and $F\in\mS$ are relatively compact and
extendable, there exist constants $\epsilon>0$ and $c_1>0$ such that for every
$F\in\mS$ and any $x, y\in\be(F)$ such that $|x-y|<\epsilon$ we have
\[\left|\frac{F(x)-F(y)}{(x-y)F'(x)}-1\right|<c_1|x-y|.\]
Taking $\epsilon$ small enough, we may assume that $c_1\epsilon<1$, then
\[0<1-c_1|x-y|<\left|\frac{F(x)-F(y)}{(x-y)F'(x)}\right|<1+c_1|x-y|\]
We also  assume that $\epsilon$ is less than the Lebesgue's number of
the covering $\mS$.

Let us show that $\epsilon$ satisfies then the conditions of
Theorem~\ref{th:visualmcriterion}. Suppose that $g_ng_{n-1}\cdots
g_1$ is a product of elements of $S$, and $F_i\in\mS$ are such that
$g_i$ is $\epsilon$-contained in $F_i$. Let $x_0, y_0$ be points on
distance less than $\epsilon$ from $\be(g_1)$. Denote $x_k=F_k\cdots
F_1(x_0)$ and $y_k=F_k\cdots F_1(y_0)$. Then
$|x_k-y_k|<\epsilon\lambda^k$.

We have
\begin{multline*}|(F_n\cdots F_1)'(x_0)|=|F_n'(x_{n-1})|\cdot
|F_{n-1}'(x_{n-2})|\cdots
|F_1'(x_0)|=\\
\left|\frac{F_n'(x_{n-1})(x_{n-1}-y_{n-1})}{x_n-y_n}\right|\cdot\left|\frac{F_{n-1}'(x_{n-2})(x_{n-2}-y_{n-2})}{x_{n-1}-y_{n-1}}\right|\cdots\left|\frac{F_1'(x_0)(x_0-y_0)}{x_1-y_1}\right|\cdot\left|\frac{x_n-y_n}{x_0-y_0}\right|.
\end{multline*}
The last product is less than
\begin{multline*}
(1-c_1|x_{n-1}-y_{n-1}|)^{-1}(1-c_1|x_{n-2}-y_{n-2}|)^{-1}\cdots
(1-c_1|x_0-y_0|)^{-1}\cdot\left|\frac{x_n-y_n}{x_0-y_0}\right|<\\
\left|\frac{x_n-y_n}{x_0-y_0}\right|
\prod_{k=0}^\infty(1-c_1\epsilon\lambda^k)^{-1},
\end{multline*}
where the infinite product is convergent.
Similarly, $|(F_n\cdots F_1)'(x_0)|$ is bigger than
\begin{multline*}
(1+c_1|x_{n-1}-y_{n-1}|)^{-1}(1+c_1|x_{n-2}-y_{n-2}|)^{-1}\cdots
(1+c_1|x_0-y_0|)^{-1}\cdot\left|\frac{x_n-y_n}{x_0-y_0}\right|>\\
\left|\frac{x_n-y_n}{x_0-y_0}\right|
\prod_{k=0}^\infty(1+c_1\epsilon\lambda^k)^{-1}.
\end{multline*}

We have shown that there exists a constant $c_2>1$ such that for any
two points $x, y$ from the $\epsilon$-neighborhood of $\be(g_1)$
we have
\[c_2^{-1}|U'(x)|<\left|\frac{U(x)-U(y)}{x-y}\right|<c_2|U'(x)|\]
where $U=F_n\cdots F_1$.
It follows that $c_2^{-2}|U'(y)|<|U'(x)|<c_2^2|U'(y)|$ for all
such $x, y$. In particular,
\[c_2^{-2}|U'(\be(g_1))|=c_2^{-2}e^{-\coc(g_n\cdots
  g_1)}<|U'(x)|<c_2^2|U'(\be(g_1))|=c_2^2e^{-\coc(g_n\cdots g_1)},\]
hence
\[c_2^{-3}e^{-\coc(g_n\cdots
  g_1)}<\left|\frac{U(x)-U(y)}{x-y}\right|<c_2^3e^{-\coc(g_n\cdots
  g_1)},\]
which implies by Theorem~\ref{th:visualmcriterion}
that $|x-y|$ is a hyperbolic metric for the cocycle $\coc$.
\end{proof}

\section{Growth and Entropy}
\subsection{Growth of graded hyperbolic groupoids}

\begin{defi}
We say that a quasi-cocycle $\psi:\G|_\X\arr\R$ is \emph{dualizable} if
there exists a quasi-cocycle $-\psi^\top$ on $\G^\top$ such that lifts
of $\psi$ and $-\psi^\top$
to the geodesic flow $\partial\G\rtimes\G$ are strongly equivalent.
\end{defi}

Here lift of a quasi-cocycle $\psi$ to the geodesic flow
$\partial\G\rtimes\G$ is given by $\psi(\xi, g)=\psi(g)$. Recall that
$\partial\G\rtimes\G$ is also equivalent to the geodesic flow of
$\G^\top$, by~\cite[Theorem~4.5.1]{nek:hyperbolic}.

Any Busemann quasi-cocycle is dualizable
by~\cite[Theorem~4.4.1]{nek:hyperbolic}. Another obvious example of a
dualizable cocycle is the constant zero cocycle. We will see later
that any \emph{H\"older continuous} cocycle is dualizable.

\begin{theorem}
\label{th:growth}
Let $\G$ be a minimal hyperbolic
groupoid graded by a Busemann quasi-cocycle $\coc:\G|_\X\arr\R$.
Let $\psi:\G|_\X\arr\R$ be a dualizable quasi-cocycle.

There exists a positive number $\Delta$ and a number
$\beta$ such that for every $x\in\X$ and for any compact neighborhood $C$ of
$\xi\in\partial\G_x$ in $\half$ there exist
positive constants $k_1$ and $k_2$ such that
\[k_1e^{\beta n}\le\sum_{g\in C\cap\G_x,\\ n-\Delta\le\coc(g)\le n}e^{\psi(g)}\le
k_2e^{\beta n}\]
for all sufficiently large $n$.
\end{theorem}

\begin{defi}
The number $\beta$ from Theorem~\ref{th:growth} is called
\emph{pressure} of $\psi$ relative to $\coc$.
\end{defi}

We are mostly interested in the case when $\psi$ is constant zero.

\begin{proof}
Let us prove at first the following technical result.
\begin{proposition}
\label{pr:mixing}
Let $(\G, \coc)$ be a minimal graded hyperbolic groupoid, and let
$\X$ be a compact topological $\G$-transversal. Then
there exists a compact generating set $S$ of $\G|_\X$ satisfying
the conditions of Proposition~\ref{pr:hyperbolicgenset}, and such that
there exists $r_0>0$ such that for every $x\in\X$ the set
$\til_x$ contains a point $\zeta_x$ such that the ball of
$\half$ of radius $r_0$ (with respect to a metric associated with the
natural log-scale on $\half$) and center in
$\zeta_x$ is contained in $\overline{T_x}$.
\end{proposition}

\begin{proof}
Let $S_1$ be a generating set satisfying conditions of
Proposition~\ref{pr:hyperbolicgenset}. Denote by
$\overline{T_g}^{(1)}$ and $\til_g^{(1)}$ the sets
defined with respect to the set $S_1$.
There exists a compact set $A\subset\G|_\X$ and a positive number
$r_1$  such that for every $g\in\G|_\X$ the set $\bigcup_{a\in
  A\cap\G_{\en{g}}}\overline{T_{ag}}^{(1)}$
contains the $r_1$-neighborhood of $\overline{T_g}^{(1)}$
(see Proposition~\ref{pr:nbhdtil}). There exists
then an integer $n>0$ such that $S=S_1\cup AS_1^n$ satisfies the
conditions of Proposition~\ref{pr:hyperbolicgenset}. Then it
will satisfy the condition of our proposition for any point
$\zeta_x\in\til_x^{(1)}\subset\til_x$.
\end{proof}

Let $\eta>0$ be such that $\coc$ and $\psi$ are
$\eta$-quasi-cocycles. Let $(S, X)$ be a generating pair of $\G$
satisfying Proposition~\ref{pr:mixing}.
Let $\mS$ be a finite covering of $S$ by positive
contracting elements of $\pG$, and let $\delta_1$ be a Lebesgue's number
of the covering $\mS$.

Let $D$ be an upper bound on values of $\coc$ on elements of $S$, and
let $\Delta=D+\eta$. Recall that $\coc(g)>2\eta$ for all $g\in S$.
Then for every product $\ldots g_2g_1$ of
elements of $S$ we have
\[\eta<\coc(g_n\cdots g_1)-\coc(g_{n-1}\cdots g_1)<\Delta.\]

Denote for $x\in\X$ and $n\ge 0$
\[L(x, n)=\{g\in T_x\;:\;n-\Delta<\coc(g)\le n\},\]
and
\[u(x, n)=\sum_{g\in L(x, n)}e^{\psi(g)}.\]

\begin{lemma}
\label{lem:ushift}
For every $k>0$ there exist positive constants $c_1, c_2$ such that
\[c_1u(x, n)\le u(x, n+k)\le c_2u(x, n)\]
for all $x\in\X$ and $n>0$.
\end{lemma}

\begin{proof}
For every element $g=g_t\cdots g_1\in L(x, n+k)$, where $g_i\in S$,
there exists an index $j<t$ such that
$h=g_j\cdots g_1\in L(x, n)$. Choose one such $h$ for every $g\in L(x,
n+k)$ and denote it $\varphi_1(g)$. There is a uniform upper bound
(depending on $k$ but not on $x$ or $n$) on the
distance from $g$ to $\varphi_1(g)$ in the Cayley graph, hence there is a
uniform upper bound $q$ on $|\varphi_1^{-1}(h)|$, and a uniform
upper bound $c_0$ on
$|\psi(g)-\psi(\varphi_1(g))|$. It follows that
\begin{multline*}
u(x, n+k)=\sum_{g\in L(x, n+k)}e^{\psi(g)}\le e^{c_0}\sum_{g\in L(x,
  n+k)}e^{\psi(\varphi_1(g))}\le\\ qe^{c_0}\sum_{h\in L(x,
  n)}e^{\psi(h)}=qe^{c_0}u(x, n).
\end{multline*}

On the other hand for every element $h=g_l\cdots g_1\in L(x, n)$ there
exists $\varphi_2(h)\in L(x, n+k)$ of the form $g_r\cdots g_{l+1}g_l\cdots
g_1$ where $r\ge l$. Again, there is a uniform upper bound on the
distance between $h$ and $\varphi_2(h)$. Hence, there are uniform upper bounds
on $|\psi(g)-\psi(\varphi_2(h))|$ and on $|\varphi_2^{-1}(h)|$, which
implies an inequality of the form $u(x, n)\le
c_1^{-1}u(x, n+k)$.
\end{proof}

\begin{lemma}
\label{lem:xy}
There exists a constant $c\ge 1$ such that
for any two points $x, y\in\X$ we have
\[c^{-1}u(y, n)\le u(x, n)\le cu(y, n)\]
for all $n\ge 0$.
\end{lemma}

\begin{proof}
The sets of values of the transformations $R_g^h$ in general do not
belong to $G|_{\X}$, since the ranges of the elements $F\in\mS$ do not
belong to $\X$. However, the ranges of $F\in\mS$ belong to a compact
set $\X'$ containing $\X$, and since $\X$ is a topological
transversal, there exists a compact set $Q\subset\G$ such that for
every $f\in R_g^h$ there exists $q_f\in Q$ such that
$q_ff\in\G|_{\X}$. Choose such $q_f$ for every $f\in T_g$, and define
\begin{equation}
\label{eq:tilder}
{\wt R}_g^h(f)=q_ff.
\end{equation}
Note that both transformations $R_g^h$ and ${\wt R}_g^h$ have the same
continuous extension onto $\til_g$.

\begin{lemma}
\label{lem:rxyn}
There is a constant $N>0$ such that for every pair $x, y\in\X$ 
there exists
$g\in T_x$ such that $R_y^g$ is defined, $\coc(g)<N$, and ${\wt R}_y^g(T_y)\subset T_x$.
\end{lemma}

\begin{proof}
Let $\delta_0$ be as in Lemma~\ref{lem:deltabijection}
Consider a finite covering $\{W_i\}_{i\in I}$ of the set $\X$ by open subsets
of diameter less than $\delta_0$.

Let a point $\zeta_x\in\til_x$ and a number $r_0$ satisfy the
conditions of Proposition~\ref{pr:mixing}. Then the $r_0$-neighborhood of
$\zeta_x=\cdots h_2h_1$ in $\half$ is contained in $\overline{T_x}$.
By minimality of $\G$,  for every $i\in I$ the set of elements
$g\in\G_x^\X$ such that $\en(g)\in W_i$
is an $\Xi$-net (with respect to the usual combinatorial metric on the
Cayley graph) in $\G(x, S)$ for some fixed $\Xi>0$ (not depending on
$x$ and $i$).

If $q\in\G|_{\X}$ is an element of length at most $\Xi$ such that
$\be(q)=\en(h_l\cdots h_1)$, and $D$ is an upper bound on the values
of $|\coc|$ on elements of $S$, then the value of $\coc$ on a geodesic
path connecting $h_l\cdots h_1$ with $qh_l\cdots h_1$ in the Cayley
graph $\G(x, S)$ is bounded below by $\coc(h_l\cdots
h_1)-(D+\eta)\Xi>l\eta-(D+\eta)\Xi$. It follows that there exists $l_0$
not depending on $x$ such that for every $l\ge l_0$ the (combinatorial)
$\Xi$-neighborhood of $h_l\cdots h_1$ in the Cayley graph $\G(x, S)$
is contained in the (hyperbolic) $r_0/2$-neighborhood of
$\zeta_x$.

If $l$ is big enough, $g$ belongs to the combinatorial
$\Xi$-neighborhood of $h_l\cdots h_1$, and $T_y^g$ is defined, then
by~\eqref{eq:dilationgf}, the hyperbolic diameter of the set of
values of ${\wt T}_y^g$ is less than $r_0/2$. 

It follows that there exists a constant $N$ such that for every $i\in
I$ there exists $g_i\in T_x$ such that $\en(g)\in W_i$, $\coc(g)<N$, the
set of values of ${\wt T}_y^g$ for every $y\in W_i$ is contained in
the $r_0$-neighborhood of $\zeta_x$, hence is contained in $T_x$.
\end{proof}

The set of elements $g$ satisfying the conditions of
Lemma~\ref{lem:rxyn} is contained in a compact set of the form
$\bigcup_{k=0}^nS^k$ for some $n$ not depending on $x$ and $y$. It
follows, by dualizability of $\coc$ and $\psi$, that the
differences $|\coc({\wt R}_y^g(h))-\coc(h)|$ and
$|\psi({\wt R}_y^g(h))-\psi(h)|$ are uniformly
bounded.

The map
$R_y^g$ is injective by Lemma~\ref{lem:deltabijection}. Consequently,
the cardinalities of the sets $({\wt R}_y^g)^{-1}(h)$ are uniformly bounded.

Consequently, using Lemma~\ref{lem:ushift} we have an estimate of
the form $u(y, n)\le c\cdot u(x, n)$.
\end{proof}

\begin{proposition}
\label{prop:multiplicative}
There exists a constant $c_2$ such that for every $x\in\X_0$ and any
positive numbers $n_1, n_2$ we have
\[c_2^{-1}u(x, n_1)u(x, n_2)\le u(x, n_1+n_2)\le c_2u(x, n_1)v(x, n_2).\]
\end{proposition}

\begin{proof}
There exists a constant $q_1$ such that every element
$g\in L(x, n_1+n_2)$ can be decomposed into a product
$g=g_1g_2$ such that $\coc(g_1)$ and $\coc(g_2)$ belong to the
intervals $[n_1-\Delta, n_1+\Delta]$ and
$[n_2-\Delta, n_2+\Delta]$ respectively, at
least in one and at most in $q_1$ ways.

It follows (using Lemmas~\ref{lem:xy} and~\ref{lem:ushift})
that $u(x, n_1+n_2)\le k_1\cdot u(x, n_1) u(x,
n_2)$ for some constant $k_1$.

On the other hand, for any pair $g_1\in L(x, n_1)$ and $g_2\in L(\en(g_1), n_2)$
we have $n_1+n_2-2\Delta-\eta<\coc(g_2g_1)<n_1+n_2+\eta$. There exists
a constant $q_2>1$
such that there exist at most $q_2$ pairs $h_1\in L(x, n_1)$ and $h_2\in
L(\en(h_1), n_2)$ such that $h_2h_1=g_2g_1$.
Hence (again using Lemmas~\ref{lem:xy} and~\ref{lem:ushift})
we have $u(x, n_1+n_2)\ge c_2^{-1}u(x, n_1)u(x, n_2)$ for some $c_2>1$.
\end{proof}

The following lemma is Exercise 99 in~\cite{pose:problems}
(next after a more famous problem on sub-additive sequences).

\begin{lemma}
\label{lem:polya}
Let $a_n$, $n\ge 1$, be a sequence of real numbers such that
$a_{n_1}+a_{n_2}-1\le a_{n_1+n_2}\le a_{n_1}+a_{n_2}+1$
for all $n_1$ and $n_2$. Then the limit $\rho=\lim_{n\to\infty}a_n/n$
exists and $n\rho-1\le a_n\le n\rho+1$ for all $n$.
\end{lemma}

Let now $c_2$ be as in Proposition~\ref{prop:multiplicative}.
Define the sequence
\[\alpha_n=\ln(u(x, n))/\ln c_2.\]
Then Proposition~\ref{prop:multiplicative}
implies that for any $n_1, n_2$, we have
\[\alpha_{n_1}+\alpha_{n_2}-1\le
a_{n_1+n_2}\le\alpha_{n_1}+\alpha_{n_2}+1,\]
and by Lemma~\ref{lem:polya} it follows that the limit
$\lim_{n\to\infty}\alpha_n/n=\rho$
exists and $n\rho-1\le\alpha_n\le n\rho+1$ for all $n$.

Consequently there exist constants $\beta$ and $k>1$ such
that
\begin{equation}
\label{eq:coneest}
k^{-1}e^{n\beta}\le u(x, n)\le ke^{n\beta}
\end{equation}
for all $n>0$ and $x\in\X$.

Let now $C$ be any compact neighborhood in
$\half$ of a point of $\partial\G_x$. Then by compactness, $C$ can be covered by
a finite number of sets of the form $\overline{T_g}$, which gives us an
upper bound of the form $k_2e^{n\beta}$ for $\sum_{g\in C\cap\G_x,
  n-\Delta\le \coc(g)\le n}e^{\psi(g)}$.
On the other hand, since the collection $\overline{T_g}$ for
$g\in\G_x$ is a basis of neighborhoods of points of $\partial\G_x$,
there exists a subset of $C$ of the form
$\overline{T_g}$, which gives us a lower bound finishing the proof of
the theorem.
\end{proof}

The following proposition is a direct corollary of
Theorem~\ref{th:growth}.

\begin{proposition}
\label{pr:poincareseries}
Let $\coc$ and $\psi$ be a Busemann and a dualizable quasi-cocycles on
$\G$, respectively.
Let $f$ be a continuous function of compact support on $\half$ not
identically equal to zero on $\partial\G_x$. Consider the series
\[\mathcal{P}_{f, \coc, \psi}(s)=\sum_{g\in\G_x^\X}f(g)e^{-s\coc(g)+\psi(g)}.\]
If $\beta$ is pressure of $\psi$ relative to $\coc$, then the series $\mathcal{P}_{f, \coc,
  \psi}(s)$
diverges for $s\le\beta$ and converges for $s>\beta$. 
\end{proposition}
\subsection{Entropy of hyperbolic groupoids and Smale quasi-flows}

\begin{defi}
Pressure of the zero cocycle $\psi(g)=0$ relative to the Busemann
cocycle $\coc$ is
called the \emph{entropy} of the graded groupoid $(\G, \coc)$ and
is denoted $h(\G, \coc)$, or just $h(\coc)$.
\end{defi}

\begin{proposition}
Entropy of a hyperbolic groupoid is positive and
\[h(\coc)=\lim_{n\to\infty}\frac{\ln |\{g\in T_x\;:\;\coc(g)\le
  n\}|}n.\]
for every $x\in X$.
\end{proposition}

\begin{proof}
It is enough to prove that entropy is positive, i.e., that sequence
$u(x, n)$ from the proof of Theorem~\ref{th:growth} is unbounded.

Suppose, by contradiction that $u(x, n)<m$ for every $n$.
Since every path $g_1, g_2g_1, g_3g_2g_1, \ldots$ connecting $x$ to a
point $\xi\in\til_x$ (where $g_i\in S$) intersects each of the sets
$L(x, n)$, and any two such paths which have infinite intersection
converge to the same point of $\til_x$, we get that $|\til_x|<m$, in
particular, that $\G^{(0)}$ has isolated points. Then we can
find a singleton that is a $\G$-transversal (by minimality of
$\G$). But groupoid of germs of a pseudogroup acting on a single point
can not satisfy the conditions of Definition~\ref{def:hyperbolic}.
\end{proof}

Let $(S, \X)$ be a generating pair of $\G$ satisfying
Definition~\ref{def:hyperbolic}. For a finite subset $N\subset\X$ denote by
$v(N, n)$ cardinality of the set of elements $g\in\G$ such that $g$ is
a product of elements of $S$, $\coc(g)\le n$, and $\be(g)\in N$.

\begin{lemma}
Supremum of the number
\[\beta_N=\lim_{n\to\infty}\frac{\ln v(N, n)}n\]
over finite subsets $N\subset\X$
is attained and is equal to the entropy $h(\coc)$.
\end{lemma}

\begin{proof}
Since $S$ can be embedded into a generating set satisfying
Proposition~\ref{pr:hyperbolicgenset},
$\beta_N$ is not greater than the entropy. It is also obvious that
$\beta_N$ does not decrease when we increase the set $N$. Consequently,
it is enough to show that there exists $N$ such that $\beta_N$ is
equal to the entropy.

Let $\xi\in\partial\G_x$, and let $C$ be a compact neighborhood of
$\xi$ in $\half$. Then there exists a finite set $A\subset\G_x^\X$ such
that $C$ is included in the set of products $\ldots g_2g_1g$ for
$g_i\in S$ and $g\in A$. It follows then that $\beta_N$ for
$N=\en(A)$ is not less than the entropy.
\end{proof}

\begin{theorem}
\label{th:equalentropies}
Entropies of $(\G, \coc)$ and $(\G, \coc)^\top$ are equal.
\end{theorem}

\begin{proof}
Let $(S, \X)$ be a generating pair of $\partial\G\rtimes\G$ satisfying
the conditions of~\cite[Definition~4.1.1]{nek:hyperbolic}, and
let $\mS$ be a covering of $S$ by rectangles.

Choosing the elements of the covering $\mS$ small enough, we may
assume that there exists $c>0$ such that for every non-empty product
$U_1\cdots U_n$ of elements of $\mS$ the values of
\[\wt\coc(g), \wt\coc(h), \coc(\proj_+(g)), \coc^\top(\proj_-(g))\]
differ from each other not more than by $c$ for all $g, h\in U_1\cdots
U_n$ (see Corollary~\ref{cor:cocdifc}).

Let $\epsilon$ be a Lebesgue's number of the covering $\mS$, and let
$\mathcal{R}$ be a covering of $\X$ by a finite number of open
rectangles of diameter less than $\epsilon$.

Since the Smale quasi-flow $\partial\G\rtimes\G$ is locally
diagonal (see~\cite[Proposition~4.7.8]{nek:hyperbolic}), we may assume that for any two rectangles $R_1,
R_2\in\mathcal{R}$ and elements $g_1, g_2\in\Gh$ such that
$\be(g_i)\in R_1$ and $\en(g_i)\in R_2$ equalities
$\proj_+(g_1)=\proj_+(g_2)$ or $\proj_-(g_1)=\proj_-(g_2)$ imply
$g_1=g_2$.
Consider localization $\Gh$ of $\partial\G\rtimes\G$ onto
$\mathcal{R}$.

Let $N_+\subset\proj_+(\Gh^{(0)})$ and $N_-\subset\proj_-(\Gh^{(0)})$ be finite
subsets such that $\beta_{N_+}$ and $\beta_{N_-}$ are equal to the
entropies of $\proj_+(\Gh)$ and $\proj_-(\Gh)$. We may
assume that $N_+$ and $N_-$ have non-empty intersections with
$\proj_+(R)$ and $\proj_-(R)$, respectively, for every $R\in\mathcal{R}$.

Let $V(N_i, n)$, for $i\in\{+, -\}$, be the set of elements $h_i\in\proj_i(\Gh)$ such that $h_i$ is a
product of elements of $\proj_i(S)$, $\coc(h_i)\le n$, and
$\be(h_i)\in N_i$. Then
$\beta_{N_i}=\lim_{n\to\infty}\ln|V(N_i, n)|/\ln n$. We will
denote $V(N_i)=\bigcup_{n\ge 1}V(N_i, n)$.

We will say that $g_+\in V(N_+)$ and $g_-\in V(N_-)$ are
\emph{related} if there exists $h\in\Gh$ such that $h$ is a product of
elements of $S$ (more precisely of their copies in the localization),
$\proj_+(h)=g_+$, and $\proj_-(h)=g_-$.

Suppose that $g_+\in V(N_+, n)$ is equal to a product
$\proj_+(s_1)\cdots\proj_+(s_k)$ of elements of $\proj_+(S)$. Let
$R_1, R_2\in\mathcal{R}$ be such that $\be(s_k)\in R_1$ and
$\en(s_1)\in R_2$. Let
$U_i\in\mS$ be such that $s_i$ is $\epsilon$-contained in $U_i$. Then
$\proj_+(\be(U_1\cdots U_n)\cap R_1)=\proj_+(R_1)$, and
$\proj_-(\en(U_1\cdots U_n)\cap R_2)=\proj_-(R_2)$.
For any $x_-\in N_-$ such that
$x_-\in\proj_-(R_2)$ we can find germs $r_1, \ldots, r_k$ of $U_1,
\cdots, U_k$ such that $\proj_+(s_1)=\proj_+(r_1)$, the product
$r_1\cdots r_k$ is defined, and $\proj_-(\en(r_1))=x_2$.
Consequently, every element $g_+\in V(N_1, n)$ is related to an
element $g_-\in V(N_-, n+c)$. By the same argument, every element
$g_-\in V(N_-, n)$ is related to an element $g_+\in V(N_1, n+c)$.

Note that by local diagonality, an element of $V(N_+)$ can not be
related to more than $|N_-|$ elements of $V(N_-)$,
and similarly, an element of $V(N_-)$ can not be related to more
than $|N_+|$ elements of $V(N_+)$.

Consequently, $|V(N_1, n)|\le |N_1|\cdot |V(N_2, n+2c)|$, and $|V(N_2,
n)|\le |N_2|\cdot |V(N_1, n+2c)|$, which implies that $\beta_{N_1}=\beta_{N_2}$.
\end{proof}

One can prove in a similar way that pressure of a cocycle $\psi$
relative to $\coc$ is equal to pressure of $\psi^\top$ relative to $\coc^\top$.

\section{Quasi-conformal measures}

\subsection{Definition and basic properties}
Let $\pG$ be a pseudogroup acting on a space $\G^{(0)}$. A Radon
measure $\mu$ on
$\G^{(0)}$ is \emph{quasi-invariant} if for every $F\in\pG$ and every
$A\subset\be(F)$ such that $\mu(A)=0$ we have $\mu(F(A))=0$.

If $\mu$ is quasi-invariant with respect to $\pG$, then for every
$F\in\pG$ we have the corresponding Radon-Nicodim derivative
\[\rho_\mu(F, x)=\frac{d F^*\mu}{d\mu}(x),\]
where $x\in\be(F)$ and $F^*\mu$ is the pull-back of $\mu$ by $F$.
Note that $\rho_\mu(F, x)$ depends only on the germ $(F, x)\in\G$.

Integrating the counting measure on $\G_x$ by $\mu$ we get a measure
$\mu_{\be}$ on $\G$ given by the formula
\[\int f(g)\;d\mu_{\be}(g)=\int\sum_{g\in\G_x}f(g)\;d\mu(\be(g)),\]
where $f:\G\arr\R$ is a compactly supported continuous
function. Similarly, we have a measure $\mu_{\en}$ on $\G$ given by
\[\int f(g)\;d\mu_{\en}(g)=\int\sum_{g\in\G^x}f(g)\;d\mu(\en(g)).\]
Quasi-invariance of $\mu$ is equivalent to absolute continuity of
$\mu_{\be}$ and $\mu_{\en}$ with respect to each other. In particular,
if $\mu$ is quasi-invariant, then
we have a well defined notion of null sets in $\G$ with respect to
$\mu$ (i.e., with respect to $\mu_{\be}$ or $\mu_{\en}$). The
Radon-Nicodim derivative $\rho_\mu(g)$ is equal to the Radon-Nicodim
derivative $\frac{d\mu_{\en}}{d\mu_{\be}}(g)$.

It is easy to see that the map $\rho_\mu:\G\arr\R_{>0}$ satisfies the multiplicative cocycle condition
\[\rho_\mu(g_1g_2)=\rho_\mu(g_1)\rho_\mu(g_2)\]
for $\mu$-almost all composable pairs.

\begin{defi}
Let $(\G, \coc)$ be a graded hyperbolic groupoid. A Radon measure
measure $\mu$ on $\G^{(0)}$ is said to be
\emph{($\G$)-quasi-conformal} if it is
$\G$-quasi-invariant, and there exists $\beta>0$
such that
\[\rho_\mu(g)\asymp e^{-\beta\cdot\coc(g)}\]
for all $g\in\G$. The number $\beta$ is called the
\emph{exponent} of the quasi-conformal measure.
\end{defi}

Note that quasi-conformality of the measure does not depend on the
choice of the quasi-cocycle (i.e., if a measure is quasi-conformal
with respect to one quasi-cocycle, then it is quasi-conformal with
respect to any strongly equivalent quasi-cocycle).

\begin{proposition}
\label{pr:extqc}
Let $(\G, \coc)$ be a minimal graded hyperbolic groupoid, where
$\coc:\G\arr\R$ is everywhere defined. Let $\X_0\subset\G^{(0)}$ be an
open subset, and suppose that there exists a
$\G|_{\X_0}$-quasi-conformal measure $\mu_0$ on $\G|_{\X_0}$. Then
there exists a $\G$-quasi-conformal measure $\mu$ on $\G^{(0)}$ of the same
exponent as $\mu_0$.
\end{proposition}

\begin{proof}
Since every open subset of $\G^{(0)}$ is a $\G$-transversal, for every
$x\in\G^{(0)}$ there exists $g\in\G$ such that $\en(g)=x$ and
$\be(g)\in\X_0$. Hence,
there exists a set $\mathcal{U}\subset\pG$ such that
$\{\en(U)\;:\;U\in\mathcal{U}\}$ is a covering of
$\G^{(0)}$, and $\be(U)\subset\X_0$ for all $U\in\mathcal{U}$.

Let $\{\varphi_U\}_{U\in\mathcal{U}}$ be a partition of unity, where
$\varphi_U:\G^{(0)}\arr [0, 1]$ is a continuous non-negative (possibly
zero) function with compact support contained in $\en(U)$.

Define then a measure $\mu$ on $\G^{(0)}$ by the formula
\[\int f(x)\;d\mu(x)=\sum_{U\in\mathcal{U}}\int
f(U(y))\varphi_U(U(y))e^{-\beta\coc(U, y)}\;d\mu_0(y),\]
where $\beta$ is the exponent of $\mu_0$, and $f:\G^{(0)}\arr\R$ is a continuous
function of compact support.

Let $h\in\G$ be such that $\be(h)\in\X_0$, and let $x=\en(h)$. We have
\begin{multline*}\frac{dh^*\mu}{d\mu_0}(x)=
\sum_{U\in\mathcal{U}}\varphi_U(x)e^{-\beta\coc(U,
  U^{-1}(x))}\frac{d(U^{-1}\circ h)^*(\mu_0)}{d\mu_0}(\be(h))\asymp\\
\sum_{U\in\mathcal{U}}\varphi_U(x)e^{-\beta\coc(U,
  U^{-1}(x))}e^{-\beta\coc(U^{-1}h)}\asymp
\sum_{U\in\mathcal{U}}\varphi_U(x)e^{-\beta\coc(h)}=e^{-\beta\coc(h^{-1})}.
\end{multline*}
It follows that for every $g\in\G$ we have
$\frac{dg^*(\mu)}{d\mu}\asymp e^{-\beta\coc(g)}$, i.e., that $\mu$
is quasi-conformal.
\end{proof}

\begin{corollary}
Let $(\G, \coc)$ be a minimal graded hyperbolic groupoid. If there
exists a quasi-conformal measure on a graded groupoid equivalent to
$(\G, \coc)$, then it exists on $(\G, \coc)$.
\end{corollary}

\begin{proposition}
The exponent of a quasi-conformal measure is equal to the entropy of
the groupoid $\G$.
\end{proposition}

\begin{proof} We will prove this proposition for the dual groupoid
  $\G^\top$.
Let $(S, \X)$ be a generating pair of $\G$
satisfying the conditions of Proposition~\ref{pr:hyperbolicgenset}.
We will realize $\G^\top$ as the groupoid $\bdry[x]$
acting on the boundary $\partial\G_x$ of a Cayley graph $\G(x,
S)$. Let $\mu$ be a $\G^\top$-quasi-invariant measure on $\partial\G_x$.

It follows from Lemma~\ref{lem:rxyn} that there exists a compact set
$Q$ such that for every $g, h\in\G_x^{\X}$ there exists $q\in Q$ such
that $qh\in T_h$, $R_g^{qh}$ is defined, and ${\wt
  R}_g^{qh}(T_g)\subset T_h$. Then
$R_g^{qh}(\til_g)\subset\til_h$. The values of $\coc^\top$ on the
germs of $R_g^{qh}$ are equal (up to a uniformly bounded additive
constant) to $\coc(h)-\coc(g)$ (see Proposition~\ref{pr:dilationestimate}).
It follows that there exists a constant $c>1$ such that
$\mu(\til_h)\ge\mu(R_g^{qh}(\til_g))\ge
c^{-1}e^{-\beta(\coc(h)-\coc(g))}\mu(\til_g)$, i.e., such that
$e^{\beta\coc(h)}\mu(\til_h)\ge c^{-1}e^{\beta\coc(g)}\mu(\til_g)$. It
follows that
\begin{equation}
\label{eq:tilgtilhmeas}
\mu(\til_g)\asymp e^{-\beta\coc(g)}
\end{equation}
for all $g$.

Let $u(x, n)$ be as in the proof of Theorem~\ref{th:growth} (for $\psi=0$).
Then for every $n$ we get a covering of $\til_x$ by at most $u(x, n)$
sets $\til_h$ such that $\mu(\til_h)\asymp e^{-\beta n}$, and
there is a constant $k>1$ such that we can find at least $k^{-1}\cdot
u(x, n)$ disjoint subsets $\til_h$ such that $\mu(\til_h)\asymp
e^{-\beta n}$. It follows that there exist a constant $c_0>1$
such that
\[c_0^{-1}u(x, n)e^{-\beta\coc(h)}\le\mu(\til_x)\le c_0u(x,
n)e^{-\beta\coc(h)}.\]
But $\mu(\til_x)$ is a constant, and
$u(x, n)\asymp e^{\beta_0 n}$, where $\beta_0$ is entropy of
$(\G, \coc)$. Consequently,
$\beta=\beta_0$. Theorem~\ref{th:equalentropies} now finishes the proof.
\end{proof}

\begin{proposition}
\label{pr:Ahlfors}
Let $\mu$ be a quasi-conformal measure on an open transversal $\X_0$
of a graded minimal hyperbolic groupoid $(\G, \coc)$. Let $|\cdot|$ be a
hyperbolic metric on $\X_0$ of exponent $\alpha$. Let $\beta=h(\coc)$. Let $\X_1\subset\X_0$ be a compact topological transversal.
Then for all $r>0$ small enough and all $x\in\X_1$ we have
\[\mu(B(x, r))\asymp r^{\beta/\alpha}.\]
\end{proposition}

\begin{proof}
It is enough to prove the proposition for any equivalent
groupoid. Consequently, we can use duality, and prove the proposition for the groupoid
$\G^\top=\bdry[x]$ instead of $\G$.

Every ball $B(\xi, r)\subset\partial\G_x$ is contained in
the set of points
$\zeta\in\partial\G_x$ such that $\ell(x, y)>\ln r/\alpha-c$ for some
constant $c>0$. Recall that $\ell(x, y)$ is the minimal value of
$\coc$ along a geodesic path in a Cayley graph of $\G$
connecting $\xi$ to $\zeta$. It follows that the ball $B(\xi, r)$
can be covered by a bounded number (not depending on $\xi$ and $r$)
of sets of the form $\til_h$, such that
$\coc(h)\doteq-\ln r/\alpha$.

On the other hand, moving $h$ along the geodesic converging to $\xi$ we
can find a set $\til_h\subset B(\xi, r)$ such that $\coc(h)\doteq-\ln
r/\alpha$.

By~\eqref{eq:tilgtilhmeas}, $\mu(\til_h)\asymp e^{-\beta\coc(h)}$.
We get then the necessary estimates from both sides to show that
$\mu(B(\xi, r))\asymp e^{\beta\ln r/\alpha}=r^{\beta/\alpha}$.
\end{proof}

\begin{corollary}
\label{cor:Hausdorff}
Every quasi-conformal measure on $\G^{(0)}$ is equivalent to the
Hausdorff measure of the hyperbolic metric of dimension $\beta/\alpha$,
where $\beta=h(\coc)$, and $\alpha$ is exponent of the
hyperbolic metric. In particular, any two quasi-conformal measures are
equivalent.
\end{corollary}

\subsection{Existence of quasi-conformal measures}

We will apply here the standard construction of the Patterson-Sullivan
measure on boundaries of hyperbolic graphs (see, for
example~\cite{coornaert:psmeasure})
to show existence of quasi-conformal measures.

Let $(\G, \coc)$ be a graded hyperbolic groupoid, and let $\G(x, S)$
be its Cayley graph. Denote by $\beta=h(\coc)$.

\begin{lemma}
\label{l:msbdd}
Let $f:\half\arr\C$ be a continuous function of compact
support. Consider the series
\[\mathcal{P}(s)=\sum_{g\in\G_x^\X}f(g)e^{-s\coc(g)}.\]
There exists a constant $c>0$ such that
$|\mathcal{P}(s)|\le c(1-e^{\beta-s})^{-1}$ for all $s>\beta$ that are
sufficiently close to $\beta$.
\end{lemma}

Recall, that by Proposition~\ref{pr:poincareseries}, the series
$\mathcal{P}(s)$ converges for all $s>\beta$.

\begin{proof}
Let $C$ be the support of $f$. By Theorem~\ref{th:growth}, there
exist positive integers $\Delta$, $k$, and $n_0$ such that
\[|\{g\in C\cap\G_x\;:\;n-\Delta\le\coc(g)\le n\}|\le ke^{\beta n}\]
for all $n\ge n_0$. Denote $L(n)=\{g\in
C\cap\G_x\;:\;n-\Delta\le\coc(g)\le n\}$. Let $c_1=\sup|f|$. We have
\begin{multline*}|\mathcal{P}(s)|\le\sum_{g\in
    C\cap\G_x}ce^{-s\coc(g)}\le\\
c_1\sum_{n<n_0}|L(n)|e^{-s(n-\Delta)}+c_1
\sum_{n\ge n_0}|L(n)|e^{-s(n-\Delta)}\le\\
c_1\sum_{n<n_0}|L(n)|e^{-s(n-\Delta)}+c_1ke^{s\Delta}\sum_{n\ge
  0}e^{(\beta-s)n}=\\
c_1\sum_{n<n_0}|L(n)|e^{-s(n-\Delta)}+c_1ke^{s\Delta}(1-e^{\beta-s})^{-1}.
\end{multline*}
There exists only a finite number of elements of $C$ such that
$\coc(g)<n_0$, consequently the first summand is continuous for all
$s$, and hence its product with $(1-e^{\beta-s})$ goes to zero as $s\to\beta$.
\end{proof}

Let measure $\mu_s$ on $\half$ for $s>\beta$ be given by
\begin{equation}\label{eq:mus}\int f\;d\mu_s=(1-e^{\beta-s})\sum f(g)e^{-s\coc(g)}.\end{equation}
for all continuous functions $f$ on $\half$ of compact support.

\begin{proposition}
\label{pr:quasiconformal}
There exists a sequence $s_k\to\beta+$ such that $\mu_{s_k}$ is weakly
converging to a measure $\mu$. The limit measure $\mu$ is supported on
$\partial\G_x$ and is quasi-conformal with respect to $(\G^\top, \coc^\top)$.
\end{proposition}

\begin{proof}
By Uniform Boundedness Principle, the set $\{\mu_s\;:\;s>\beta\}$ is
bounded in the space dual to the space of continuous compactly
supported functions on $\half$. Hence, by Banach-Alaoglu Theorem,
there exists a sequence $s_k$ such that $s_k\to\beta+$ and $\mu_{s_k}$
is weakly converging to a measure on $\half$.

Since for every
$g\in\G_x^\X$ (i.e., for any point $g\in\half\setminus\partial\G_x$)
we have $(1-e^{\beta-s})e^{-s\coc(g)}\to 0$ as $s\to\beta$, the
support of $\mu$ is contained in $\partial\G_x$. Suppose that
$f:\half\arr\R$ is a continuous non-negative compactly supported
function, and let $f(\xi)>0$ for a point $\xi\in\partial\G_x$. Then
for any positive number $p$ less than $f(\xi)$ there exists a compact
neighborhood $C_0$ of $\xi$ in $\half$ such that $f(x)>p$ for all
$x\in C_0$. Let $\Delta$ be as in Theorem~\ref{th:growth}. Denote by
$L(n)$ the set of elements $g\in C_0$ such that $n-\Delta\le\coc(g)\le
n$. Then the size of the set $L(n)$ is bounded below by $k_0e^{\beta
  n}$ for some $k_0>0$ and for all $n\ge n_1$ for some $k_1$. Every
point of $C_0$ belongs to at most $\Delta+1$ sets $L(n)$. Consequently,
\begin{multline*}\int f\;\mu_s=(1-e^{\beta-s})\sum
  f(g)e^{-s\coc(g)}\ge(1-e^{\beta-s})\sum_{g\in C_0}pe^{-s\coc(g)}\ge\\
(1-e^{\beta-s})k_0p(1+\Delta)^{-1}\sum_{n\ge n_1}e^{n\beta}e^{-sn}=\\
(1-e^{\beta-s})k_0p\Delta^{-1}e^{(\beta-s)n_1}(1-e^{(\beta-s)})^{-1}=
k_0p\Delta^{-1}e^{(\beta-s)n_1},
\end{multline*}
hence $\int f\;d\mu>0$, and support of $\mu$ coincides with $\partial\G_x$.

Let ${\wt R}_h^g$ be partial transformations of $\half$ defined by~\eqref{eq:tilder}.
The values of $\coc^\top$ on germs of $\wt R_h^g$ on $\partial\G_x$ differ from
$\coc(g)-\coc(h)$ by a uniformly bounded constant not depending on $g$
or $h$ see~\eqref{eq:dilationgf}.

Taking $\delta_0$ in the definition of maps $R_h^g$ sufficiently
small, and using Corollary~\ref{cor:cocdifc}, we get an estimate
$\coc({\wt R}_h^g(x))\doteq\coc(g)-\coc(h)$ for all $x\in T_h$, whence
\[e^{-s\coc({\wt R}_h^g(x))}\asymp
e^{-s\coc(x)}e^{-s(\coc(g)-\coc(h))}\asymp e^{-s\coc(x)}e^{-s\coc^\top(\gamma)}\]
for every $x\in T_h$ and for any germ $\gamma$ of $R_h^g$ on
$\til_h$. It follows that for every subset $A\subset T_h$ we have
\[\mu_s({\wt R}_h^g(A))\asymp e^{-s\coc^\top(\gamma)}\mu_s(A),\]
where $\gamma$ is any germ of $R_h^g$ on $\til_h$.
Consequently, the measure $\mu$ is $(\G^\top, \coc^\top)$-quasi-conformal.
\end{proof}

\section{Continuous cocycles}

\subsection{General definitions}

\begin{defi}
Let $\G$ be a topological groupoid. A map $\coc:\G\arr\R$ is a
\emph{continuous cocycle} if it is continuous and
$\coc(gh)=\coc(g)+\coc(h)$ for all $(g, h)\in\G^{(2)}$.
\end{defi}

An \emph{orbispace} is an equivalence class of a proper 
groupoid of germs. A groupoid $\G$ is said to be 
\emph{proper} if the map $\be\times\en:\G\arr\G^{(0)}\times\G^{(0)}$
is proper, i.e., if for this map preimages of compact sets are compact.

Note that if a groupoid of germs $\G$ is proper and principal (i.e., if all
isotropy groups $\G_x^x$ are trivial), then it is equivalent to the
trivial groupoid (a groupoid without non-unit elements) on the space
of $\G$-orbits.

A \emph{flow} (resp.\ a \emph{$\Z$-action}) on an orbispace is an
equivalence class (in the sense of~\cite{muhlyrenault:equiv}) of a proper
groupoid of germs $\G$ together with an action $F_t$ of $\R$ (resp.\ of $\Z$)
on $\G^{(0)}$ such that for every $g\in\G$ and every $t\in\R$ (resp.\
$\in\Z$) there exists a unique element $g_t\in\G$ such that the germs
$F_t\circ g$ and $g_t\circ F_t$ are equal.
The corresponding topological flow is given by the action of $\R$
(resp.\ of $\Z$) on the space of orbits of $\G$.

\begin{proposition}
\label{pr:smaleflow}
Every Smale quasi-flow $(\Gh, \coc)$ with a continuous cocycle $\coc$
is equivalent to a flow on an orbispace.
\end{proposition}

\begin{proof}
Consider the space $\Gh^{(0)}\times\R$ and the action of $\Gh$ on it
defined by the rule:
\[(g, t)\cdot h=(gh, t+\coc(h)).\]
The obvious action of $\R$ on $\Gh^{(0)}\times\R$ commutes with the
defined action of $\Gh$. We get then commuting actions on
$\Gh^{(0}\times\R$ of $\Gh$ and of
the groupoid $\wh{\Gh}$ generated by the action of $\Gh$ and $\R$.
It follows from~\cite[Theorem~4.3.1]{nek:hyperbolic} and condition (5)
of~\cite[Definition~4.1.1]{nek:hyperbolic} that these actions are
proper. They are also obviously free. It follows that they define
equivalence between $\Gh$ and $\wh{\Gh}$ in the sense of~\cite{muhlyrenault:equiv}.

It also follows from properness of the action of $\Gh$ on $\Gh^{(0)}\times\R$ that the groupoid
of this action together with the natural action of $\R$ define an
orbispace flow.
\end{proof}

\begin{defi}
Two continuous cocycles $\coc_1$ and $\coc_2$ on $\G$ are
\emph{co-homologous} if there exists a continuous function
$\phi:\G^{(0)}\arr\R$ such that
\[\coc_1(g)-\coc_2(g)=\phi(\en(g))-\phi(\be(g))\]
for all $g\in\G$.
\end{defi}

Let $\coc:\G\arr\R$ be a continuous cocycle, let $f:Y\arr\G^{(0)}$ be
an \'etale map, and let $\G|_f$ be the corresponding
localization. Then the \emph{lift} $\coc_f$ of $\coc$ to $\G|_f$ is given
by $\coc_f(x, g, y)=\coc(g)$, where $(x, g, y)$ is a lift of $g$. It
is easy to see that $\coc_f$ is a continuous and well defined cocycle.

Recall that two groupoids are equivalent if and only if they have
isomorphic localization (see a remark just after Definition~\ref{def:equivalent}).

\begin{defi}
We say that two continuous cocycles $\coc_1:\G_1\arr\R$ and
$\coc_2:\G_2\arr\R$ defined on equivalent groupoids are
\emph{continuously equivalent} if there exists a common localization
of $\G_1$ and $\G_2$ such that the lifts of $\coc_1$ and $\coc_2$ to it
are cohomologous.
\end{defi}

It is not hard to see that two continuously equivalent graded
groupoids define topologically conjugate flows. Choice of a particular
graded groupoid in a continuous equivalence class correspond in some
sense to the choice of a ``generalized transversal'' of the flow.

\subsection{H\"older continuous cocycles}

\begin{defi}
Let $\G$ be a groupoid of germs preserving Lipschitz class of a metric
$|\cdot|$ on $\G^{(0)}$.
We say that a cocycle $\coc:\G\arr\R$ is \emph{H\"older
continuous} if there exists $p>0$ such that
for every $g\in\G$ there exists a neighborhood
$U\in\pG$ of $g$ such that $\coc(U, x)$ is $p$-H\"older continuous
as a function of $x$ with respect to $|\cdot|$.
\end{defi}

\begin{proposition}
\label{pr:projcoc}
Let $\Gh$ be a Smale quasi-flow and let $\psi:\Gh\arr\R$ be a
H\"older continuous cocycle. Then there exists an equivalent groupoid
$\Gh'$ and a H\"older continuous cocycle $\psi':\Gh'\arr\R$
continuously equivalent to $\psi$, and a unique, up to cohomology,
H\"older continuous cocycle $\psi_+:\proj_+(\Gh)\arr\R$
such that $\psi_+(\proj_+(g))=\psi'(g)$ for all
$g\in\Gh'$.
\end{proposition}

\begin{proof}
Possibly passing to a localization, we assume that there exist sets
$\X, \X_1, S, \mS$ satisfying the conditions of
Proposition~\ref{prop:generatorsflow}. We will prove our proposition
for restriction of the groupoid onto $\X$.

Choose a point $x_B\in B^\circ$ for every rectangle
$A\times B$. For a point $y\in A^\circ\times B^\circ$ choose a sequence $F_1, F_2,
\ldots$ of elements of $\mS^{-1}$ such that
$\proj_-(\en(F_n))\subset\proj_-(\be(F_{n+1}))$ and $y\in\be(F_n\cdots
F_1)$ for all $n$. Let $y'=[y, x_B]$. Then $\{y, y'\}\subset\be(F_n\cdots
F_1)$ for all $n$ and $|F_n\cdots F_1(y)-F_n\cdots
F_1(y')|\asymp\lambda^n |y-y'|$
for some $\lambda\in (0, 1)$. Define then
\begin{multline*}\phi(y)=\lim_{n\to\infty}\psi(F_n\cdots F_1,
  y)-\psi(F_n\cdots F_1, y')=\\
\sum_{n=1}^\infty(\psi(F_n, F_{n-1}\cdots F_1(y))-\psi(F_n,
F_{n-1}\cdots F_1(y'))).\end{multline*}
Note that $|\psi(F_n, F_{n-1}\cdots F_1(y))-\psi(F_n, F_{n-1}\cdots
F_1(y'))|\le c\lambda^{pn} |y-y'|^p$ for some $c, p>0$, by the
H\"older continuity of $\psi$. It follows that the limit exists.

Let $y_1, y_2\in A^\circ\times B^\circ$ such that $\proj_+(y_1)=\proj_+(y_2)$. Then
\[\phi(y_1)-\phi(y_2)=\sum_{n=1}^\infty(\psi(F_n, F_{n-1}\cdots
F_1(y_1))-\psi(F_n, F_{n-1}\cdots F_1(y_2)));\]
and since $|\psi(F_n, F_{n-1}\cdots F_1(y_1))-\psi(F_n, F_{n-1}\cdots
F_1(y_2))|\le c\lambda^{pn}|y-y'|^p$, the function $\phi$ is
$p$-H\"older continuous on each slice $\proj_+^{-1}(x)$.

Let us show that $\phi$ does not depend on the choice of the sequence
$F_i$. Let $F_i'\in\mS$ be another sequence. Then there exist strictly
increasing sequences $n_i$ and $m_i$, a finite set $\mathcal{A}$ of
relatively compact elements of $\wt{\Gh}$, and a sequence
$U_i\in\mathcal{A}$ such that
\[(U_iF_{n_i}\cdots F_1, y)=(F_{m_i}'\cdots F_1', y),\quad
(U_iF_{n_i}\cdots F_1, y')=(F_{m_i}'\cdots F_1', y')\]
for all $i$.

Then
\begin{multline*}\psi(F_{m_i}'\cdots F_1', y)-\psi(F_{m_i}'\cdots F_1', y')=
\psi(U_iF_{n_i}\cdots F_1, y)-\psi(U_iF_{n_i}\cdots F_1, y')=\\
\psi(U_i, F_{n_i}\cdots F_1(y))-\psi(U_i, F_{n_i}\cdots F_1(y'))+
\psi(F_{n_i}\cdots F_1, y)-\psi(F_{n_i}\cdots F_1, y')\end{multline*}
Since $|F_{n_i}\cdots F_1(y)-F_{n_i}\cdots F_1(y')|\asymp
\lambda^{n_i}|y-y'|$, the difference $\psi(U_iF_{n_i}\cdots F_1,
y)-\psi(U_iF_{n_i}\cdots F_1, y')$ goes to zero by H\"older continuity
of $\psi$. Consequently, the limit $\phi(y)$ defined in
terms of $F_i$ is the same as the one defined in terms of $F_i'$.

\begin{lemma}
Define $\psi'(g)=\psi(g)-\phi(\be(g))+\phi(\en(g))$. Then for any two
$g_1, g_2\in\Gh$ such that $\proj_+(g_1)=\proj_+(g_2)$ we have
$\psi'(g_1)=\psi'(g_2)$.
\end{lemma}

\begin{proof}
Consider two sequences $F_1, F_2, \ldots$ and $G_1, G_2, \ldots$ of
elements of $\mS$ such that
$\proj_-(\en(F_n))\subset\proj_-(\be(F_{n+1}))$,
$\proj_-(\en(G_n))\subset\proj_-(\be(G_{n+1}))$; and
 $\be(g_1), \be(g_2)\in\be(F_1)$, $\en(g_1), \en(g_2)\in\be(G_1)$. Let
$\mathcal{A}$ be as above. Then there exist increasing sequences $n_i$
and $m_i$, a sequence $U_i\in\mathcal{A}$ and $\epsilon>0$ such that
$G_{n_i}\cdots G_1g_1(F_{m_i}\cdots F_1)^{-1}$ is $\epsilon$-contained
in $U_i$. It follows that for all $i$ big enough we have
\[F_{m_i}\cdots F_1(\be(g_2))\in\be(U_i).\] Then
$\proj_+(U_i, F_{m_i}\cdots F_1(\be(g_2)))=\proj_+(G_{n_i}\cdots
G_1g_1(F_{m_i}\cdots F_1)^{-1})$, as $U_i$ is a rectangle. But
$\proj_+(G_{n_i}\cdots G_1g_2(F_{m_i}\cdots
F_1)^{-1})=\proj_+(G_{n_i}\cdots G_1g_2(F_{m_i}\cdots F_1)^{-1})$. By
local diagonality of $\Gh$, if two elements of $\Gh$ have the same
source and equal projections, then they are equal (we assume that the
rectangles $A^\circ\times B^\circ$ are small enough). Consequently, $G_{n_i}\cdots
G_1g_2(F_{m_i}\cdots F_1)^{-1}\in U_i$. It follows then that
$\psi(G_{n_i}\cdots G_1g_1(F_{m_i}\cdots F_1)^{-1})-\psi(G_{n_i}\cdots
G_1g_2(F_{m_i}\cdots F_1)^{-1})\to 0$ as $i\to\infty$. We have
\begin{multline*}
\psi'(g_1)-\psi'(g_2)=\\
\psi(g_1)-\psi(g_2)-\phi(\be(g_1))+\phi(\be(g_2))+\phi(\en(g_1))-\phi(\en(g_2))
=\\
\psi(g_1)-\psi(g_2)+\lim_{i\to\infty}\psi(F_{m_i}\cdots F_1,
\be(g_2))-\psi(F_{m_i}\cdots F_1, \be(g_1))+\\
\lim_{i\to\infty}\psi(G_{n_i}\cdots G_1, \en(g_1))-\psi(G_{n_i}\cdots
G_1, \en(g_2))=\\
\lim_{i\to\infty}\psi(G_{n_i}\cdots G_1g_1(F_{m_i}\cdots
F_1)^{-1})-\psi(G_{n_i}\cdots G_1g_2(F_{m_i}\cdots F_1)^{-1})=0.
\end{multline*}
\end{proof}

Let us show that $\phi$ is  H\"older continuous, which will imply
H\"older continuity of $\psi'$. Let $\epsilon$ be a Lebesgue's number of
the covering $\mS$ of $S$. Suppose that $F$ is
$\Lambda$-Lipschitz for every $F\in\mS$.
If $|y_1-y_2|<\epsilon\Lambda^{-k}$, then there
exists a sequence $F_1, F_2, \ldots, F_k\in\mS^{-1}$ such that
$\proj_-(\en(F_n))\subset\proj_-(\be(F_{n+1}))$ for all $1\le n\le
k-1$, and $y_1, y_2\in\be(F_k\cdots F_1)$.

Continuing the sequence $F_1, \ldots, F_k$ to two sequences $F_i$ and
$F_i'$ for $y_1$ and $y_2$, respectively, and repeating the estimates
in the proof of uniqueness of $\phi$, we get an estimate
$|\phi(y_1)-\phi(y_2)|<C\lambda^k$, which implies that $\phi$ is
H\"older continuous.

It remains to show that $\psi'$ can be projected onto the groupoid
$\proj_+(\Gh)$, i.e., that there exists a H\"older continuous cocycle
$\psi_+:\proj_+(\Gh)\arr\R$ such that $\psi_+(\proj_+(g))=\psi'(g)$
for every $g\in\Gh$.

Every element of $\proj_+(\Gh)$ is equal to a product
$\proj_+(s_1)\cdots\proj_+(s_n)$ where $s_i\in\mS\cup\mS^{-1}$.
Let us define
\[\psi_+(\proj_+(s_1)\cdots\proj_+(s_n))=\psi'(s_1)+\cdots+\psi'(s_n).\]
We have to show that $\psi_+$ is well defined. It is enough to show
that if $\proj_+(s_1)\cdots\proj_+(s_n)$ is a unit, then
$\psi'(s_1)+\cdots+\psi'(s_n)=0$.

By~\cite[Lemma~4.4.2]{nek:hyperbolic}
there exist sequences $g_i, h_i, r_i, r_i'\in\Gh$ such that
$s_i=g_{i-1}^{-1}r_ig_i$, $\proj_+(r_i)=\proj_+(r_i')$, and
$r_1'\cdots r_n'g_n$ is composable. Then
\[\proj_+(s_1)\cdots\proj_+(s_n)=\proj_+(g_0^{-1})\proj_+(r_1'\cdots
r_n'g_n),\] and since this product is a unit, we have
$\proj_+(g_0)=\proj_+(r_1'\cdots r_n'g_n)$, hence
$\psi'(g_0)=\psi'(r_1'\cdots r_n'g_n)$.

But then
$\sum\psi'(s_i)=\sum\psi'(g_{i-1}^{-1}r_ig_i)=\psi'(g_0^{-1})+\psi'(g_n)+
\sum\psi'(r_i)=\psi'(g_0^{-1})+\psi'(g_n)+\sum\psi'(r_i')=
\psi'(g_0^{-1})+\psi'(r_1'\cdots
r_n'g_n)=0$.

H\"older continuity of $\psi_+$ follows from H\"older continuity of
$\psi'$ and the fact that every element of $\proj_+(\Gh)$ is a product
$\proj_+(g)\proj_+(h)$ for $g, h\in\Gh$. Uniqueness of $\psi_+$ is
straightforward.
\end{proof}

If $\psi_+:\proj_+(\Gh)\arr\R$ satisfies the conditions of the last
proposition for a cocycle $\psi:\Gh\arr\R$, then we say that $\psi_+$ is a
\emph{projection} of $\psi$.

Note that if $\psi:\G\arr\R$ is an arbitrary H\"older continuous cocycle
on a hyperbolic groupoid, then its lift $\wt\psi(\xi, g)=\psi(g)$ to
the geodesic
flow $\partial\G\rtimes\G$ is H\"older continuous and $\psi=\wt\psi_+$ is
projection of $\wt\psi$. According to Proposition~\ref{pr:projcoc}
there is a projection $\psi^\top:\G^\top\arr\R$ of $-\wt\psi$ onto
$\proj_-(\partial\G\rtimes\G)$. We call
$\psi^\top=\proj_+(-\wt\psi)=\proj_-(\wt\psi)$ the \emph{dual
  cocycle} for the cocycle $\psi$. The dual cocycle is H\"older
continuous and is uniquely defined up to continuous equivalence.

In particular, we have the following corollary of
Proposition~\ref{pr:projcoc}.

\begin{corollary}
Every H\"older continuous cocycle on a hyperbolic groupoid is dualizable.
\end{corollary}

We have the following explicit description of the dual cocycle
$\psi^\top:\bdry[x]\arr\R$.

\begin{proposition}
\label{pr:dualcocyclecont}
Let $\G$ be a minimal hyperbolic groupoid, and let $\psi:\G\arr\R$ be
a H\"older continuous cocycle. 

Define a map $\rho:T_g\arr\R$ by
\[\rho(f)=\psi(R_g^h(f))-\psi(f).\]
Then $\rho(f)$ converges uniformly to $\psi^\top(R_g^h, \xi)$
as $f\to\xi$.
\end{proposition}

\begin{proof}
Let $\xi\in\partial\G_x$ be equal to $\ldots g_2g_1\cdot g$ for
$g_i\in S$. Let $U_i\in\mS$ be such that $g_i$ is $\epsilon$-contained
in $U_i$. It follows from the proof of Proposition~\ref{pr:projcoc}
that cocycle is given by
\[\psi^\top(R_g^h, \xi)=\lim_{n\to\infty}\psi(U_n\cdots
U_1h)-\psi(U_n\cdots U_1g).\]
Note that $R_g^h(U_n\cdots U_1g)=U_n\cdots U_1h$, so we have
\[\psi^\top(R_g^h, \xi)=\lim_{n\to\infty}\psi(R_g^h(U_n\cdots
U_1g)-\psi(U_n\cdots U_1g)).\]

Let $A\subset\G|_{\X}$ be as in Proposition~\ref{pr:nbhdtil}. Then for
every $n$ there exists $a\in A$ such that $\overline{T_{ag_n\cdots
    g_1g}}$ is a neighborhood of $\xi$. Since $\xi$ is an internal
point of $\til_g$, for all $n$ big enough we have
$\overline{T}_{ag_n\cdots g_1g}\subset\overline{T}_g$. Let
$\mathcal{A}$ be a finite covering of $A$ by bi-Lipschitz elements of
$\pG$. Let $\zeta=\ldots h_2h_1ag_n\cdots
g_1g\in\overline{T}_{ag_n\cdots g_1g}$ (where the sequence $h_i$ is
  finite or infinite), and suppose that
$U\in\mathcal{A}$ and $V_i\in\mS$ are such that $a$ and $h_i$ are
$\epsilon$-contained in $U$ and $V_i$. Then, by Lemma~\ref{lem:deltabijection},
we have
\[R_g^h(\zeta)=\ldots V_2V_1UU_n\cdots U_1h.\]

Since $\psi$ is H\"older, we may assume that $\mS$ is such that
there exist constants $c_1$ and $p$ such
that $|\psi(V_i, x)-\psi(V_i, y)|\le c_1|x-y|^p$ for all $V_i\in\mS$,
and $x, y\in\be(V_i)$.

Since $V_i$ and $U_i$ are contracting, and the elements of $\mathcal{A}$ are
bi-Lipschitz, there exist $c>1$ and $\lambda\in (0, 1)$ such that
\begin{multline*}
\bigl|\psi(U_n\cdots U_1h)-\psi(U_n\cdots U_1g)-\\
(\psi(V_m\cdots V_1UU_n\cdots U_1h)-\psi(V_m\cdots V_1UU_n\cdots U_1g))\bigr|=\\
|\psi(V_m\cdots V_1U, \en(U_n\cdots U_1h))-\psi(V_m\cdots V_1U,
\en(U_n\cdots U_1g))|\le\\
\sum_{i=1}^m\left|\psi(V_i, \en(V_{i-1}\cdots V_1UU_n\cdots U_1h))-\psi(V_i,
  \en(V_{i-1}\cdots V_1UU_n\cdots U_1g))\right|\le\\
\sum_{i=1}^mc\lambda^{n+i}|\en(h)-\en(g)|
\le\frac{c|\en(h)-\en(g)|}{1-\lambda}\lambda^n.
\end{multline*}
It follows that $\psi(R_g^h(f))-\psi(f)$ uniformly converges to
$\psi^\top(R_g^h, \xi)$ when $f\to\xi$.
\end{proof}

\begin{examp}
Dual cocycle to the cocycle $\coc(F, z)=-\ln|F'(z)|$, where $F$
belongs to the pseudogroup generated by a complex
rational function, was studied in~\cite[Section~3.4]{kaimlyubich:laminations}.
\end{examp}

\subsection{Conformal measures}

\begin{defi}
Let $\coc:\G\arr\R$ be a H\"older continuous Busemann cocycle on a
hyperbolic groupoid. A Radon measure $\mu$ on $\G^{(0)}$ is
$\coc$-conformal if
\[\rho_\mu(g)=e^{-\beta\coc(g)}\]
for every $g\in\G$, where $\beta=h(\coc)$.
\end{defi}

\begin{proposition}
\label{pr:contequivdensity}
Suppose that $\coc_1:\G_1\arr\R$ and $\coc_2:\G_2\arr\R$ are
continuously equivalent continuous cocycles. If there exists
a $\coc_1$-conformal measure on $\G^{(0)}_1$, then there exists a
$\coc_2$-conformal measure on $\G_2^{(0)}$.
\end{proposition}

\begin{proof}
Suppose that the Radon-Nicodim derivative of $\mu$ is
$e^{-\beta\coc}$ and let $\coc_1$ be cohomologous to $\coc$. If
$\varphi$ is the corresponding function such that
$\coc_1(g)=\coc(g)+\varphi(\en(g))-\varphi(\be(g))$, then the measure
$\mu_1$ given by $e^{\varphi(x)}\;d\mu(x)$ satisfies
$\rho_{\mu_1}(g)=e^{-\beta\coc_1(g)}$.

It remains to prove that a $\coc$-conformal measure on a hyperbolic
groupoid $(\G, \coc)$ exists if and only if it exists for its
localization.

Let $\X_0\subset\G^{(0)}$ be an open subset, and let $\mu_0$ be a
$\coc$-conformal measure on $\X_0$. Repeating the proof of
Proposition~\ref{pr:extqc} for the case of a conformal measure, we
note that we get strict equalities everywhere instead of
estimates. Consequently, conformal measures on open subsets are
uniquely extended to conformal measures on the whole unit space. In
the other direction, a conformal measure on $\G^{(0)}$ restricted to
an open subset $\X_0$ is conformal with respect to the restriction of
the groupoid. This implies immediately that a localization of $\G$ has
a conformal measure if and only if $\G$ has a conformal measure.
\end{proof}

\begin{theorem}
\label{th:conformaldensity}
Let $(\G, \coc)$ be a minimal hyperbolic groupoid graded by a
H\"older continuous Busemann cocycle. Then there exists a unique, up to a
multiplicative constant, $\coc$-conformal measure on $\G^{(0)}$.
\end{theorem}

\begin{proof}
As usual, we will prove the theorem for the groupoid $(\G^\top,
\coc^\top)$, and then use duality.
Let $\X\subset\G^{(0)}$ be a compact topological transversal, and let
$\phi:\G^{(0)}\arr\R$ be a non-negative continuous function with non-empty
compact support that is a subset of $\X$. Let $\beta=h(\G, \coc)$.
Define for $s>\beta$ a measure $\mu_s$ on $\G_x$ by the equality
\begin{equation}\label{eq:muscon}\int
  f\;d\mu_s=(1-e^{\beta-s})\sum_{g\in\G_x}f(g)\phi(\en(g))e^{-s\coc(g)}.\end{equation}
Note that support of $\mu_s$ is a subset of $\G_x^{\X}$, since
$\phi(\en(g))=0$ for $g\notin\G_x^{\X}$.

The same arguments as in Proposition~\ref{pr:quasiconformal} show that
there exists a sequence $s_k\to\beta+$ such that $\mu_{s_k}$ weakly
converge to a measure $\mu$ supported on $\partial\G_x$.

Let us show that $\mu$ is $\coc^\top$-conformal. Let $S$ and $\mS$
satisfy the conditions of Proposition~\ref{pr:hyperbolicgenset}. Consider a germ
$(R_h^g, \xi)$ of the 
transformation $R_h^g:\overline T_h\arr\G_x\cup\partial\G_x$, where
$\xi$ is an internal point of $\til_h$, and $g, h\in\G_x$ are such that $\en(g)$ and $\en(h)$ are
sufficiently close to each other.

Recall that $\phi$ is uniformly continuous on $\G^{(0)}$, $f$ is
continuous on $\half$, and there exist $c>0$ and
$\lambda\in (0, 1)$ such that
$|\en(R_h^g(r))-\en(r)|<c\lambda^{\coc(r)}$ for all
$r\in T_h$. Let $C\subset\half$ be a compact neighborhood of $\xi$,
and denote $C_0=C\cap\G_x$. We have
\[
\frac{\int_{R_h^g(C)}f(r)\;d\mu_s(r)}{\int_C f(R_h^g(r))\;d\mu_s(r)}=
\frac{\sum_{s\in
    C_0}f(R_h^g(r))\phi(\en(R_h^g(r)))e^{-s\coc(R_h^g(r))}}{\sum_{s\in C_0}f(R_h^g(r))\phi(\en(r))e^{-s\coc(r)}}.
\]
As we make the neighborhood $C$ converge to $\xi$, the difference
$|\phi(\en(R_h^g(r))-\phi(\en(r))|$ uniformly converges to $0$, while
the difference $\coc(R_h^g(r))-\coc(r)$ uniformly converges to
$\coc^\top(R_h^g, \xi)$, by
Proposition~\ref{pr:dualcocyclecont}. The functions $\phi$ and
$f$ are bounded, the series $\sum_{r\in C}e^{-s\coc(r)}$ converges and
has an upper bound not depending on $C$ (see Lemma~\ref{l:msbdd}). Consequently,
\[
\frac{\int_{R_h^g(C)}f(r)\;d\mu_s(r)}{\int_C
  f(R_h^g(r))\;d\mu_s(r)}\to e^{-s\coc^\top(R_h^g, \xi)}
\]
as $C$ converges to $\xi$. It follows that $\mu$ is $\coc^\top$-conformal.

It remains to prove uniqueness of a conformal measure. It is enough to
prove uniqueness of a conformal measure on $\partial\G_x$ for
some $x\in\G^{(0)}$. Assume that $S$ satisfies the conditions of
Proposition~\ref{pr:mixing}. Let $|\xi-\zeta|$ be a metric on
$\partial\G_x$ of exponent $\alpha$ associated with the cocycle
$\coc$. We will denote by $B(\xi, r)$ the ball of radius $r$ with
center in $\xi$ in $\partial\G_x$.

Fix a number $\delta_0$ that is small enough to satisfy the conditions
of Lemma~\ref{lem:deltabijection}.

By Lemma~\ref{lem:rxyn}, there
exists a constant $N>0$ such that for every $h_1, h_2\in\G_x$ there
a transformation of the form $R_{h_1}^{gh_2}$ such that
$R_{h_1}^{gh_2}(\til_{h_1})\subset\til_{h_2}$, $0<\coc(g)<N$, and $|\en(gh_2)-\en(h_1)|<\delta_0$. 
For every $\xi_1\in\partial\G_x$ there exists $g_1$ such that
$B(\xi_1, 1)\subset\til_{g_1}$ (the proof is similar to the proof
of~\cite[Proposition~3.4.4]{nek:hyperbolic} (see also
Proposition~\ref{pr:nbhdtil} of our paper).

It follows now from Propositions~\ref{pr:mixing}
that for all $r_1, r_2>0$ and $\xi_1,
\xi_2\in\partial\G_x$ there exists a map of the form $R_{g_1}^{g_2}$
such that $R_{g_1}^{g_2}(\overline{B(\xi_1, r_1)})\subset B(\xi_2,
r_2)$, $\coc(g_2)-\coc(g_1)$ differs from $-\frac{\ln r_2-\ln
  r_1}{\alpha}$  by a uniformly bounded constant, and
$|\en(g_1)-\en(g_2)|<\delta_0$.

Fix $r_1\in (0, 1)$, $\xi_1\in\partial\G_x$, and choose for every $r_2\in (0,
1)$ and $\xi_2\in\partial\G_x$ a transformation
$R_{g_1}^{g_2}$, and denote $V_{\xi_2, r_2}=\{\xi_2\}\cup
R_{g_1}^{g_2}(B(\xi_1, r_1))$. Let $\mathcal{V}_{\xi_1, r_1}$ be the set of all
sets of the form $V_{\xi, r}$ for $\xi\in\partial\G_x$ and $r\in (0,
1)$. Then $\mathcal{V}_{\xi_1, r_1}$ is a covering of $\partial\G_x$ by closed sets.

It follows from Proposition~\ref{pr:dualcocyclecont}
and the fact that elements of $\mS$ are $\lambda$-contractions for
some fixed $\lambda$, that there exists a constant $L$
depending only on $\mS$ such that for any germ $(R_{g_1}^{g_2},
\zeta)$ we have \[|\coc^\top(R_{g_1}^{g_2},
\zeta)-(\coc(g_2)-\coc(g_1))|<L|\en(g_1)-\en(g_2)|\le L\delta_0.\]

By conformality of $\mu_1$ and $\mu_2$ we get
 \begin{multline*}\frac{\mu_1(V_{\xi, r})}{\mu_2(V_{\xi, r})}=\frac{\int_{V_{\xi,
       r}}e^{-\beta\coc^\top(R_{g_1}^{g_2}, \zeta)}\;d\mu_1(\zeta)}{\int_{V_{\xi,
       r}}e^{-\beta\coc^\top(R_{g_1}^{g_2}, \zeta)}\;d\mu_2(\zeta)}\le\\
 \frac{e^{-\beta(\coc(g_2)-\coc(g_1))}e^{L\beta\delta_0}\cdot
\mu_1(B(\xi_1,
r_1))}{e^{-\beta(\coc(g_2)-\coc(g_1))}e^{-L\beta\delta_0}\cdot\mu_2(B(\xi_1,
r_1))}=e^{2L\beta\delta_0}\frac{\mu_1(B(\xi_1, r_1))}{\mu_2(B(\xi_1, r_1))}.
\end{multline*}

We also conclude that there
exist positive constants $c_1, c_2$ such that
\[c_1e^{-\beta(\coc(g_2)-\coc(g_1))}\le\mu_i(R_{g_1}^{g_2}(\til_{g_1}))\le
c_2e^{-\beta(\coc(g_2)-\coc(g_1))},\] for all transformations $R_{g_1}^{g_2}$ and for all $i=1,2$.
It follows, by Proposition~\ref{pr:Ahlfors}, that there exist constants $c_3,
c_4$ such that
\[c_3e^{\beta\ln r/\alpha}=
c_3r^{\beta/\alpha}\le\mu_i(V_{\xi, r})\le c_4e^{\beta\ln
  r/\alpha}=c_4r^{\beta/\alpha}.\]

We will use a version of Vitali's covering theorem given
in~\cite[Theorem~2.8.7]{federer:geommeas}. Fix a constant $\tau>1$ and
denote for a set $V_{\xi, r}\in\mathcal{V}_{\xi_1, r_1}$ by  $\wh{V}_{\xi, r}$ the union
of all the set of the form $V_{\zeta, s}\in\mathcal{V}_{\xi_1, r_1}$ such that
$V_{\zeta, s}\cap V_{\xi, r}$ is
non-empty, and $s\le\tau r$. Then $\wh{V}_{\xi, r}\subset B(\xi,
r+2\tau r)$, since diameter of $V_{\zeta, s}$ is not greater than
$2s\le 2\tau r$. It follows then from Proposition~\ref{pr:Ahlfors}
that there exists a constant $c_5>0$ such that
\[\mu_i(\wh{V}_{\xi,
  r})<c_5(1+2\tau)^{\beta/\alpha}r^{\beta/\alpha}\le
\frac{c_5(1+2\tau)^{\beta/\alpha}}{c_3}\mu_i(V_{\xi, r}).\]
Consequently, the conditions
of~\cite[Theorem~2.8.7]{federer:geommeas} are satisfied for the
covering $\mathcal{V}_{\xi_1, r_1}$ of $\partial\G_x$. Consequently, for any open
subset $W\subset\partial\G_x$ there exists a set
$\mathcal{W}$ of pairwise disjoint elements of $\mathcal{V}_{\xi_1, r_1}$ such that
$\bigcup\mathcal{W}\subset W$ and
$\mu_i(W\setminus\bigcup\mathcal{W})=0$.

It follows from Proposition~\ref{pr:Ahlfors} that $\mu_i$ are 
doubling measures (see~\cite[Section~1.4]{heinonen}), hence they
satisfy Lebesgue's differentiation
theorem~\cite[Theorem~1.8]{heinonen}:
\begin{equation}\label{eq:lebesgue}
\lim_{r\to 0}\frac{1}{\mu_i(B(\xi, r))}\int_{B(\xi, r)}
f\;d\mu_i=f(\xi)
\end{equation}
for almost all $\xi$ and for all locally integrable functions $f$.

The measures $\mu_1, \mu_2$ are mutually absolutely continuous by
Corollary~\ref{cor:Hausdorff}. The Radon-Nicodim derivative
$d\mu_1/d\mu_2$ is constant on $\partial\G^\top$-orbits, by
conformality. Suppose that $d\mu_1/d\mu_2$ is not constant on $\partial\G_x$. Then there
exist $m_1<m_2$ and sets of non-zero measure $A_1, A_2\subset\partial\G_x$
such that $d\mu_1/d\mu_2$ is less than $m_1$ on $A_1$ and bigger than
$m_2$ on $A_2$. There exists $\xi_1\in A_1$ such that
\[\lim_{r\to 0}\frac{\mu_1(B(\xi_1, r))}{\mu_2(B(\xi_1, r))}=
\lim_{r\to 0}\frac{1}{\mu_2(B(\xi_1, r))}\int_{B(\xi_1,
  r)}\frac{d\mu_1}{d\mu_2}\;d\mu_2=\frac{d\mu_1}{d\mu_2}(\xi_1)<m_1.\]

It follows for every $\epsilon>0$ there exists $r_1$ such that
$\frac{\mu_1(B(\xi_1, r))}{\mu_2(B(\xi_1, r))}$ is less than
$m_1+\epsilon$ for all $r\ge r_1$. Consider then the covering
$\mathcal{V}_{\xi_1, r_1}$. For every
$V\in\mathcal{V}_{\xi_1, r_1}$ we have
$\frac{\mu_1(V)}{\mu_2(V)}<e^{2L\delta_0}(m_1+\epsilon)$. Since every
open subset of $\partial\G_x$ can be represented as a countable union
of disjoint elements of $\mathcal{V}_{\xi_1, r_1}$ and a zero-set, for
every open set $W\subset\partial\G_x$ we have
$\frac{\mu_1(W)}{\mu_2(W)}\le e^{2L\delta_0}(m_1+\epsilon)$. But
$\epsilon$ and $\delta_0$ can be made arbitrarily small. Consequently,
$\frac{\mu_1(W)}{\mu_2(W)}\le m_1$ for all open sets $W$, which is a
contradiction with the inequalities $\lim_{r\to 0}\frac{\mu_1(B(\xi_2,
  r))}{\mu_2(B(\xi_2, r))}>m_2>m_1$.
\end{proof}

\begin{examp}
\label{ex:ratfunction}
Let $f(z)\in\C(z)$ be a hyperbolic rational function, i.e., a complex
rational function expanding on a neighborhood of its Julia set
$J_f$. Then $f:J_f\arr J_f$ is a local homeomorphism, and
the pseudogroup $\wt{\mathfrak{F}}$ generated by it is hyperbolic. We
have seen (in Section~\ref{s:visual}) that a
Busemann cocycle $\coc:\mathfrak{F}\arr\Z$ is given by the formula
\[\coc((f^n, x)^{-1}\cdot (f^m, y))=n-m.\]
Note that $\coc$ is locally constant, hence H\"older continuous.

The conformal measure associated with $\coc$ is the weak limit of
uniform distributions $\mu_n$ on the sets $f^{-n}(z_0)$ for any fixed
$z_0\in J_f$. It is also the measure of maximal entropy of the
dynamical system $(f, J_f)$, and is known (in the general setting of
not necessarily hyperbolic functions) as
Brolin-Lyubich measure~\cite{lyubich:measure}.

The inverse limit $\mathcal{S}$ of the constant sequence of maps $f:J_f\arr J_f$
together with the homeomorphism on it induced by $f$ is a Smale
space called the \emph{natural extension} of $f$.
The groupoid $\mathfrak{F}$ is projection of this Smale space
onto the unstable direction of the natural local product
structure. The properties of the natural extensions (also in general,
not only in the hyperbolic case), including their measure theory were
studied in~\cite{lyubichminsk,kaimlyubich:laminations}.

The cocycle
\[\coc_1(F, z)=-\ln|F'(z)|\]
is another natural Busemann cocycle on $\mathfrak{F}$. Restriction
onto $J_f$ of the usual metric on $\C$ is a hyperbolic metric of
exponent 1 associated with $\coc_1$, by Proposition~\ref{pr:natural}. Note that $\coc_1$ is smooth,
hence H\"older continuous. Measures conformal with respect to the
cocycle $\coc_1$ where defined for any complex rational function by
D.~Sullivan in~\cite{sullivan:LN}. It would be interesting to extend
theory of hyperbolic groupoids to a more general setting, so that
it will include all rational functions acting on the Julia set and all
Kleinian groups acting on the limit set, see~\cite{sullivan:density}
(and not only geometrically finite groups without parabolic elements,
as it is now).

The $\coc_1$-conformal measure is, by
Corollary~\ref{cor:Hausdorff}, equivalent to the Hausdorff measure on
the Julia set. In particular, the Hausdorff dimension of the Julia
set is equal to the critical exponent of the series
\[\sum_{n\ge 0}\sum_{z\in f^{-n}(z_0)}e^{-s\ln|(f^{\circ n})'(z)|}=
\sum_{n\ge 0}\sum_{z\in f^{-n}(z_0)}|(f^{\circ n})'(z)|^{-s}.\]
These results (existence of the conformal measure and the formula for the
Hausdorff dimension) are partial cases
of~\cite[Theorem~1.2]{MR2001m:37089} due to C.~McMullen.
\end{examp}

\subsection{Invariant measure on the flow}

Let $(\G, \coc)$ be a hyperbolic groupoid with a H\"older continuous
Busemann cocycle. Suppose that $\G^{(0)}$ is a disjoint union of rectangles.

Let $\coc_i$ for $i=+,-$ be cocycles cohomologous to $\coc$ such that
$\proj_i(\coc_i)$ are well defined on $\proj_i(\G)$. Let
$\phi_i:\G^{(0)}\arr\R$ be such that
\[\coc(g)-\coc_i(g)=\phi_i(\en(g))-\phi_i(\be(g)).\]
\[\coc_i(g)=\coc(g)-\phi_i(\en(g))+\phi_i(\be(g)).\]

Consider then on each rectangle the direct product $\mu_+\times\mu_-$
of the conformal
measures defined by projections of $\coc_i$ and $-\coc_i$ onto the corresponding
directions. We have
\begin{multline*}
\rho_{\mu_+\times\mu_-}(g)=\exp(-\beta\coc_+(g)+\beta\coc_-(g))=\\
\exp(-\beta(\coc(g)-\phi_+(\en(g))+
\phi_+(\be(g))-\coc(g)+\phi_-(\en(g))-\phi_-(\be(g))))=\\
\exp(-\beta((\phi_-(\en(g))-\phi_+(\en(g)))-(\phi_-(\be(g))-\phi_+(\en(g))))).
\end{multline*}
It follows that the measure $\mu$ on $\G^{(0)}$ given by
\[\int f(x)\;d\mu(x)=\int
f(x)e^{-\beta(\phi_-(x)-\phi_+(x))}\;d\mu_+\times\mu_-(x)\]
is invariant with respect to $\G$.

We can extend this invariant measure to any groupoid equivalent to $\G$, using
the same methods as in Propositions~\ref{pr:extqc} and~\ref{pr:contequivdensity}.

\begin{examp}
In the case when the cocycle $\coc:\G\arr\R$ has values in $\Z$, the
Smale quasi-flow $\G$ is equivalent to a \emph{Smale
  (orbi)space}. The corresponding invariant measure is the classical
Bowen measure, see~\cite{bowen:periodicmeasures,ruelle:therm}, which
is usually constructed using Markov partitions.

In general, for a continuous cocycle $\coc:\G\arr\R$ on a Smale quasi-flow, the groupoid
$\G$ is equivalent to a \emph{Smale flow} on an orbispace (see
Proposition~\ref{pr:smaleflow}) and
the constructed measure is a direct generalization of the
Bowen-Margulis measure for Anosov flows, constructed in~\cite{margulis:measure}
and~\cite{bowen:periodicflow}. This follows from the scaling
properties of the corresponding measures on the stable and unstable
foliations, invariance under holonomies (i.e., definition of a
$\coc$-conformal measure), and the uniqueness statement of Theorem~\ref{th:conformaldensity}.
See also~\cite{bowenmarcus}.

Note that the geodesic flow on a negatively curved compact manifold
$M$ is equivalent (as a topological groupoid) to the action of the fundamental
group $\pi_1(M)$ on the square $\partial\wt\M\times\partial\wt\M$ of
the ideal boundary of the universal covering of $M$, minus the
diagonal. It follows that the groupoid generated by the geodesic flow
is equivalent to the geodesic flow $\partial\G\rtimes\G$ of the hyperbolic
groupoid $\G$ of the action of the fundamental group $\pi_1(M)$ on its
Gromov boundary (equivalently, on $\partial\wt\M$). We obtain in this
way the well known fact that the Bowen-Margulis measure associated
with a geodesic flow on a negatively curved compact manifold $M$ can be
obtained from the Patterson-Sullivan measure on $\partial\wt\M$.
See the paper of D.~Sullivan~\cite{sullivan:measure}, for the constant
curvature case, and the paper of
V.~Kaimanovich~\cite{kaimanovich:geodesicflow} for the general
case. Note that it is also shown in the latter paper that the
Patterson-Sullivan measures are Hausdorff measures of naturally
defined metrics. The paper~\cite{kaimanovich:geodesicflow}
also considers measures and metrics arising from different
choices of the cocycle.

For more on relations between hyperbolic geometry, Busemann cocycles,
and conformal measures, see the monograph~\cite{kaimlyubich:laminations}. 
In particular, it studies conformal measures on the groupoid $\mathfrak{F}$ generated by a
complex rational function and on its dual $\mathfrak{F}^\top$.
\end{examp}

\end{document}